\documentclass[final]{siamart190516}

\usepackage{amsfonts}
\usepackage{amsmath}
\usepackage{graphicx}
\usepackage{epstopdf}
\usepackage{algorithmic}
\usepackage{bbm}
\usepackage{tikz}
\usetikzlibrary{arrows}
\usepackage{enumerate}
\ifpdf

  \DeclareGraphicsExtensions{.eps,.pdf,.png,.jpg}
\else
  \DeclareGraphicsExtensions{.eps}
\fi

\makeatletter
\def\namedlabel#1#2{\begingroup
    #2%
    \def\@currentlabel{#2}%
    \phantomsection\label{#1}\endgroup
}
\makeatother

\newsiamremark{remark}{Remark}
\newsiamremark{hypothesis}{Hypothesis}
\crefname{hypothesis}{Hypothesis}{Hypotheses}
\newsiamthm{claim}{Claim}
\newsiamthm{assumption}{Assumption}
\headers{}{}
\title{Blackwell optimality and policy stability for long-run risk sensitive stochastic control}

 \author{
Nicole B{\"a}uerle\thanks{Department of Mathematics, Karlsruhe Institute of Technology (KIT), Karlsruhe, Germany (\email{nicole.baeuerle@kit.edu}).}
 \and
Marcin Pitera\thanks{Institute of Mathematics, Jagiellonian University, Krakow, Poland
  (\email{marcin.pitera@uj.edu.pl});  research supported by NCN grant 2020/37/B/ST1/00463.}
  \and
{\L}ukasz Stettner\thanks{Institute of Mathematics, Polish Academy of Sciences, Warsaw, Poland
  (\email{l.stettner@impan.pl});  research supported by NCN grant 2020/37/B/ST1/00463.}}

\usepackage{amsopn}
  
\DeclareMathOperator{\ent}{\textnormal{Ent}}  
\usepackage{enumitem} 

\DeclareMathOperator*{\argmax}{arg\,max}
\def\cF{\mathcal{F}}

\def\bP{\mathbb{P}}
\def\bE{\mathbb{E}}
\def\bR{\mathbb{R}}
\def\bQ{\mathbb{Q}}

\def\bN{\mathbb{N}}

\def\cP{\mathcal{P}}
 
\DeclareMathOperator{\dif}{d \!}

\makeatletter
\def\namedlabel#1#2{\begingroup
    #2%
    \def\@currentlabel{#2}%
    \phantomsection\label{#1}\endgroup
}
\makeatother

\newtheorem{example}{Example}[section]

\begin{document}

\maketitle

% REQUIRED
\begin{abstract}
This paper analyzes the stability of optimal policies in the long-run stochastic control framework with an averaged risk-sensitive criterion for discrete-time MDPs on finite state-action space. In particular, we study the robustness of optimal controls when perturbations to the risk-aversion parameter are applied, and investigate the Blackwell property, together with its link to the risk-sensitive vanishing discount approximation framework. Finally, we present examples that help to better understand the intricacies of the risk-sensitive control framework.
\end{abstract}

% % REQUIRED

 %REQUIRED

\section{Introduction}
Consider a classical discrete-time Markov Decision Process (MDP) long-run framework with a running reward functional, see~\cite{BauRie2011}. The {\em entropic risk sensitive} averaged criterion, together with its discounted version, is arguably the most popular MDP objective function choice if one is interested in the determination of the optimal control policy taking into account performance risk. For basic results about risk-sensitive optimal controls we refer to \cite{howard1972risk,Jaq1973,Whi1990,BraCavFer1998,DiMSte1999,CavFer2000,CavHer2009,chavez2015continuity,CavHer2017,CavCru2017,Ste2023,Ste2024,PitSte2024} and references therein; see also \cite{AsiJas2017,bauerle2014more,BisBor2023,bauerle2024markov} for extensions and alternatives.

In a nutshell, the entropic risk sensitive criterion is a non-linear extensions of the classical risk-neutral mean criterion in which the risk sensitivity parameter $\gamma\neq 0$ is encoded in the entropic utility $\frac{1}{\gamma}\ln \bE\left[e^{\gamma Z} \right]$ applied to the cumulative reward $Z$; in the averaged setup, the limit of averaged entropic utility over time is considered. The popularity of the entropic criterion could be justified by the fact that this mapping satisfies many useful properties linked to additivity, strong-time consistency (tower rule property), certainty equivalent representation, or dual relative entropy based characterisation, see \cite{KupSch2009,BiePli2003} for details. Furthermore, when $\gamma \to 0$, we obtain in the limit the classical risk-neutral criterion in which entropic utility is replaced by the expectation. Of course, the optimal policy depends on the corresponding risk-aversion parameter $\gamma$ choice or the optimal policy determination logic making the framework sensitive to eventual aversion changes or policy approximation schemes. 

In this paper, we analyse the stability of the averaged optimal policies. We restrict our study to MDPs with finite state and action space (as typically done in practical implementations) and assume that for each decision rule the implied Markov chain is irreducible. Apart from the direct analysis of risk-aversion specification stability, we are also interested in the {\em Blackwell property} of the approximating discounted strategies. This property is important, for both risk-neutral and risk-sensitive setups, in the context of theoretical algorithms based on the vanishing-discount approach as well as many Machine Learning methods based on approximating discounting schemes, see \cite{CavHer2011,PorCavCru2023,BisBor2023,DewDunEshGalRoo2020} and references therein. In particular, we want to emphasize that, in the risk-neutral context, the Blackwell optimality is often used to ensure connection between discounting and averaging schemes and constitutes usage of many dedicated Reinforced Learning algorithms based on discounting approximations; see \cite{DewGal2022,GraPet2023} and references therein. Consequently, better understanding of Blackwell optimality in the risk-sensitive setup could help boost development of efficient risk-sensitive reinforced learning algorithms which is an area of on-going and intensive research, see \cite{FeiYanWan2021,BasMaSheXu2022,DinJinLav2023,NooBar2021} and references therein. 

For completeness, let us  shortly comment on {\em the results obtained in this paper} and provide more background on the Blackwell property. On the first hand, as far as the sensitivity w.r.t.\ $\gamma$ is concerned, there are some previous studies focusing mainly on small $\gamma$ parameter analysis. For example, in \cite{chavez2015continuity} the authors consider risk-sensitive MDP with denumerable state space, bounded cost function and average cost criterion. They show that under a general form of the simultaneous Doeblin condition the value function is continuous in $\gamma$, also at $\gamma=0$. The authors give an example that this  may not be true under weaker recurrence assumptions. In \cite{Ste2023}, a more general 
certainty equivalent control problem of discrete time Markov processes with average utility functional is studied. In this setting  continuity of the optimal value function with respect to the risk parameter is shown. That being said, we are unaware of a direct finite-specification oriented studies or summaries in this area; note that even in the denumerable case, the problems could get much more complex so one expects some results to hold only in the finite state-action setup.  Apart from recalling optimal value continuity results, and providing the link between risk-sensitive stochastic control and multiplicative Poisson equations, we prove that the real line of risk-sensitivity parameters can be divided into a countable number of closed intervals on which in the interior of each interval there is a unique  (up to a symmetry) optimal risk-sensitive control with the possible exception in a countable number of points. This result is obtained by identifying the maximal value of the problem with the largest eigenvalue of a matrix using the Perron-Frobenius Theorem; see Theorem~\ref{th:MPE.closed} for details. Moreover, this implies that if there is a unique optimal stationary policy for the risk-neutral setting, this policy is also optimal for risk-sensitivity parameters that are close to zero.

On the other hand, while there are numerous studies on the Blackwell optimality within the risk-neutral finite setup, no analogous results  for the risk-sensitive case, due to our best knowledge, are available. For completeness, let us provide a short summary on the  Blackwell optimality for the risk-neutral case. This concept was initially introduced in~\cite{Bla1962}: A policy $\pi$ is said to be Blackwell optimal, if there exists $\beta_0<1$ such that $\pi$ is $\beta$-discount optimal for any $\beta\in (\beta_0,1)$. The Blackwell optimal policies are average reward optimal and it is possible to show that in MDPs with finite state and action space a Blackwell optimal policy always exists. This is usually shown by using Laurent series expansions and Cramer's rule to prove that each entry of the inverse matrix $(I-\beta \bP^u)^{-1}$ is a rational function of $\beta$ where $\bP^u$ is the transition matrix under decision rule $u.$ This approach also implies that the interval $[0,1)$ for the discount factor can be divided into closed intervals where on each interval there is one decision rule defining the optimal stationary policy for these discount factors. For further information about the early stages of Blackwell optimality and extensions, see \cite{DekHor1988,hordijk2002blackwell,cavazos1999direct,lewis2019uniform}.  Unfortunately, one cannot easily apply the aforementioned tools to the risk-sensitive case. This is linked to the fact that risk-sensitive discounted optimal policies are typically non-stationary (in contrast to risk-neutral discounted policies) and the problem becomes non-linear, so one cannot rely directly on the Laurent series expansion based techniques. That saying, in this paper we were able to show that  a form of the Blackwell property, adjusted to the non-stationary setup, can be recovered; the proofs are based on a combination of vanishing discount and span-contraction arguments. Indeed, suppose that for a fixed risk-sensitivity parameter $\gamma\neq 0$ the optimal $\beta$-discounted non-stationary Markov policy is given by $(u_0^\beta, u_1^\beta, u_2^\beta,\ldots)$. Then, for each $n\in\bN$, there is a $\beta_n<1$ such that for any $\beta\in(\beta_n,1)$ the stationary Markov policy $(u_n^\beta,u_n^\beta,\ldots )$ is optimal for the risk-sensitive average problem. For details, we refer to Theorem~\ref{th:vanishing.discount} and Theorem~\ref{th:blackwell.sensitive}, where the key results linked to vanishing discount approximation and risk-sensitive Blackwell property are stated; these theorems could be seen as the key contribution of this paper.

Our paper is organized as follows: In Section~\ref{S:introduction} we introduce the model, state our assumptions, discuss the solution technique and summarize some properties of the entropic utility which are helpful in our case. Then, in Section \ref{sec:riskaversion_parameter}, we investigate the dependence of the optimal averaged stationary policies on the risk-sensitivity parameter. In Section \ref{sec:blackwell} we focus on the vanishing discount framework and discuss Blackwell optimality in the risk-sensitive setting. Finally, we illustrate our findings with some examples given in Section \ref{S:examples}.

\section{Preliminaries}\label{S:introduction}
Let us consider a time-homogeneous Markov Decision Process (MDP) in finite settings. Namely, we fix $k,l\in\bN$ and consider the finite state space $E=\{x_1,\ldots,x_k\}$ and the finite action space $U=\{\tilde u_1,\ldots,\tilde u_l\}$. For $a\in U$, we use $P_{a}:=[p_a(x_i,x_j)]_{i,j=1}^{k}$ to denote the transition probability matrix under action $a$. For simplicity, we assume that for any $x\in E$ the action set is full, i.e., every action is always attainable; this assumption could be easily relaxed if required, see \cite{BauRie2011}.

As usual, we define the canonical space $(\Omega,\cF)=(E^{\bN},(2^E)^{\bN})$ with filtration $(\cF_n)$, where $\cF_n=(2^E)^n\times\{\emptyset,E\}^\bN$. For any policy $\pi\in \Pi$, where $\Pi$ is the set of all (attainable) policies $\pi=(a_n)_{n\in\bN}$ such that $a_n\colon\Omega\to U$ is $\cF_n$-measurable, we use $\bP^{\pi}$ to denote the corresponding controlled probability on $(\Omega,\cF)$; note that $\bP^{\pi}$ could be built directly from the transition matrix family $(P_{a})_{a\in U}$, see~\cite{BauRie2011}. For any $x\in E$ and $\pi\in U$, we also use $\bP_x^{\pi}$ to denote the controlled probability for the ($\pi$-controlled) canonical process $(X_n)$ with $X_0=x$ and $\bE_x^{\pi}$ to denote the corresponding expectation operator. Next, we associate (time-homogeneous) stationary Markov policies with functions $u\colon E\to U$. Namely, we link $u$ to the policy $\pi=(a_n)_{n\in\bN}\in \Pi$, where $a_n=u(X_n)$, and write $\bP^u$ rather than $\bP^{\pi}$. We denote the set of all such functions by $\Pi'$ and sometimes, with slight abuse of notation, write $u\in \Pi$ rather than $u\in\Pi'$; note that $|\Pi'|=l^k<\infty$. Similarly, for non-stationary Markov policies, we use the notation $(u_n)_{n\in\bN}$, where $u_n\colon E\to U$, link those policies to $\pi=(a_n)_{n\in\bN}\in \Pi$, where $a_n=u_n(X_n)$, and use $\Pi''$ to denote the set of all such policies. Finally, for consistency, given the (single step) action $a\in U$ and states $x,y\in E$, we often write $\bP^{a}(x,y)$ rather than $p_a(x,y)$, and use $\bP^{(\pi,n)}$ to denote the $n$th iteration of the controlled transition kernel linked to policy $\pi\in\Pi$.

In this paper, we consider a long-run MDP problem with a running reward function $c\colon E\times U\to\bR$. The objective function, for a pre-specified risk-averse parameter $\gamma\in\bR$, control strategy $\pi=(a_n)_{n\in\bN}\in\Pi$, and starting point $x\in E$ is equal to the {\it long run average risk-sensitive criterion} given by
\begin{equation}\label{eq:RSC}
J_\gamma(x,\pi):=
\begin{cases}
\liminf_{n\to\infty}\tfrac{1}{n}\frac{1}{\gamma }\ln\bE_x^{\pi}\left[e^{\gamma\sum_{i=0}^{n-1}c(X_i,a_i)}\right],& \textrm{if } \gamma\neq 0,\\
\liminf_{n\to\infty}\frac{1}{n}\bE_x^{\pi}\left[\sum_{i=0}^{n-1}c(X_i,a_i)\right],& \textrm{if }\gamma=0.
\end{cases}
\end{equation}
The value of $\gamma\in\bR$ determines our risk preferences. For $\gamma=0$, we are in the additive risk-neutral regime, while for 
$\gamma\neq 0$ we are in the multiplicative risk-sensitive regime. Furthermore, $\gamma<0$ corresponds to the risk-averse regime, while $\gamma>0$ denotes the risk-seeking regime, see~\cite{Whi1990}. Our goal is to maximise \eqref{eq:RSC}, i.e. find a strategy $\pi^*$ such that $J_{\gamma}(x,\pi^*)=J^*(x)$, for $x\in E$, where
\begin{equation}\label{eq:RSC.opt}
J^*_\gamma(x) :=\sup_{\pi\in \Pi}J_{\gamma}(x,\pi)
\end{equation}
is the maximal objective function value. Effectively, in this paper we consider the mixing framework under which the optimal value \eqref{eq:RSC.opt} is independent of the starting point, see \cite{BraCavFer1998} for more details.

\subsection{Ergodic assumptions}\label{S:ergodic}
In order to efficiently maximize \eqref{eq:RSC}, we need to impose specific assumptions. For simplicity, most of the results in this paper are formulated under ergodic mixing assumptions. That is, we assume the following:

\medskip

\begin{enumerate}
\item[(\namedlabel{A.1}{A.1})] (One-step mixing.) We have that
\[
\Delta:=\sup_{x,x'\in E}\sup_{A\in 2^E}\sup_{a,a'\in U} \left|\bP^a(x,A)-\bP^{a'}(x',A)\right| <1.
\]

\item[(\namedlabel{A.2}{A.2})] (Multi-step transition equivalence.) There exists $N\in\bN$ such that
\[
K:=\sup_{x,x'\in E}\sup_{y\in E}\sup_{\pi\in\Pi''} \frac{\bP^{(\pi,N)}(x,y)}{\bP^{(\pi,N)}(x',y)} <\infty.
\]
\end{enumerate}

\medskip

In \eqref{A.2}, we use convention $\tfrac{0}{0}:=1$ and $\tfrac{1}{0}:=\infty$. We decided to assume both \eqref{A.1} and \eqref{A.2} to simplify the narrative and present the results in a condensed (one-step) form. While in the finite state and action setting, the multi-step transition equivalence implies multi-step mixing for a fixed policy, \eqref{A.1} ensures one-step contraction property of the risk-sensitive averaged Bellman operator, see Proposition~\ref{eq:Bellman.op}. We also note that \eqref{A.1} could be linked to the {\it uni-chain} condition, i.e. the existence of a unique (controlled) recurrent class, which is a sufficient assumption for the risk-neutral control problems. On the other hand, for risk-sensitive problems, it might be insufficient for specific reward functions, see~\cite{CavHer2009}. That being said, for any specific function $c$, we know that \eqref{eq:RSC.opt} can be solved when $\gamma$ is sufficiently close to 0, see~\cite{PitSte2024}. Furthermore, in all results presented in this paper, \eqref{A.2} could be replaced with a stronger yet more intuitive assumption linked to state communication:

\medskip

\begin{enumerate}
\item[(\namedlabel{A.2'}{A.2'})] (Communicating states.) For each Markov policy $u\in\Pi''$ and $n\in\bN$, the $n$th iteration of the controlled transition kernel $\bP^{(\pi,n)}$ is irreducible, that is, for any $x\in E$ and $y\in E$ there exists an integer $m(x,y)\equiv m$ such that
\[
(\bP^{(\pi,n)})^m(x,y)>0.   
\]
\end{enumerate}

In an algebraic context, \eqref{A.2'} could be restated using the notion of {\it strong primitivity} introduced in \cite{CohSel1982}; this assumption ensures mixing on the whole space which indicates that the (unique) recurrent class is equal to $E$. Also, the number of steps $m$ in \eqref{A.2'}, could be effectively bounded from above by $2^k-2$ as will be shown in the proof of Proposition~\ref{pr:strong.mixing}; this bound is sharp, as proved in \cite{CohSel1982}. We refer to \cite{WuZhu2023} and references therein for other properties of matrix families and their products (that could be linked to non-homogeneous Markov chains). See also \cite{BauRie2011} for more information about communicating states assumption in the stationary MDP setup.  Now, let us show that \eqref{A.2'} is indeed stronger than \eqref{A.2} in the finite setup.

\begin{proposition}\label{pr:strong.mixing} Assume \eqref{A.2'}. Then \eqref{A.2} holds.
\end{proposition}
\begin{proof}
Assume \eqref{A.2'}. First, let us show that for any $\pi\in\Pi''$, the corresponding transition matrix $\bP^{(\pi,N)}$ is positive for $N:=2^{k}-2$, i.e. for $x,y\in E$ we have
\begin{equation}\label{eq:strong1}
\bP^{(\pi,N)}(x,y)>0.
\end{equation}
Fix $\pi=(u_1,u_2,\ldots)\in\Pi''$, $x\in E$, and set $S_n:=\{y\in E: \bP^{(\pi,n)}(x,y)>0\}$, $n\in\bN$, with $S_0=\{x\}$. Note that if $S_n=E$, then $S_{n+1}=E$, for $n\in\bN$, as $1$-step transition matrices induced by stationary Markov policies are irreducible due to \eqref{A.2'} and therefore do not contain columns consisting of all zero entries. Thus, assuming \eqref{eq:strong1} does not hold, the set $S_n$ must be a proper subset of $E$, for any $n=1,\ldots,N$. Furthermore, from \eqref{A.2'} we get that, for any $0\leq n_1<n_2\leq N$, the $(n_2-n_1)$-step transition matrix $\bP^{(u_{n_1+1},u_{n_1+2},\ldots,u_{n_2})}$ is irreducible. Thus, there must be no two equal elements in the family of sets $\{S_0,S_1,\ldots,S_N\}$; the existence of such elements, say $S_{n_1}$ and $S_{n_2}$, implies reducibility of  $\bP^{(u_{n_1+1},u_{n_1+2},\ldots,u_{n_2})}$ as $S_{n_1}=S_{n_2}$ and $S_{n_1}\neq E$. Now, since $2^k-2$ is the number of proper non-empty subsets of $E$ and $N+1>2^k-2$, we get $S_N=E$. Since $N$ does not depend on the choice of $x\in E$, this proves \eqref{eq:strong1}.

Second, note that $\bP^{(\pi,N)}(x,y)$ depends only on the choice of $x,y\in E$ as well as the first $N$ coordinates of the underlying Markov strategy $\pi\in \Pi''$. Consequently, since $|E|=k<\infty$ and $|\Pi'|^N\leq (l^k)^N<\infty$, we get that $\inf_{x,y\in E}\inf_{\pi\in\Pi''}\bP^{(\pi,N)}(x,y)>0$, which directly implies \eqref{A.2}.
\end{proof}
Finally, we note that in the finite setting, assumption \eqref{A.2} is effectively ensuring one-step recurrence or (global) transience in the finite number of steps.

\begin{corollary}
Assumption \eqref{A.2} holds if and only if there exists $N\in \bN$ such that for any $\pi\in\Pi''$ and $y\in E$, we have the following implication
\begin{equation}\label{eq:dust.to.dust}
\left(\exists_{x\in E} \bP^{(\pi,N)}(x,y)=0\right) \quad\Rightarrow \quad  \left(\forall_{x\in E} \bP^{(\pi,N)}(x,y)=0\right).
\end{equation}
\end{corollary}
\begin{proof}
($\Rightarrow$) Assume \eqref{A.2} with $N\in\bN$. Fix any $\pi\in \Pi''$ and $y\in E$ such that there exist $x'\in E$ for which $\bP^{(\pi,N)}(x',y)=0$. Then, we must have $\bP^{(\pi,N)}(x,y)=0$ for any $x\in E$ as otherwise $K\not<\infty$ due to convention $\tfrac{1}{0}=\infty$. ($\Leftarrow$) Assume there exists $N\in\bN$ for which \eqref{eq:dust.to.dust} is satisfied. Let $D(\pi):=\{(x,y)\in E\times E\colon \bP^{(\pi,N)}(x,y)>0\}$ for any $\pi\in\Pi''$, and fix $\epsilon:=\inf_{\pi\in\Pi''}\inf_{(x,y)\in D}\bP^{(\pi,N)}(x,y)$. Using similar logic as in the last part of Proposition~\ref{eq:strong1} proof, we get that $\epsilon>0$, as $\epsilon$ is an infimum over a finite set of positive numbers. From \eqref{eq:dust.to.dust}, we know that $K\leq \max(\tfrac {1}{\epsilon}, \tfrac{0}{0})<\infty$, which concludes the proof.
\end{proof}

\subsection{Bellman equation}
For a fixed $\gamma\in\bR$ we call a pair $(w(\cdot,\gamma),\lambda(\gamma))$, where $w(\cdot,\gamma): E\to \bR$ and $\lambda(\gamma)\in\bR$, a solution to the long-run risk sensitive averaged Bellman equation with risk aversion $\gamma$, if the pair solves the Bellman equation given by
\begin{equation}\label{eq:Bellman1}
\begin{cases}
\textstyle w(x,\gamma)=\max_{a\in U}\left[c(x,a)-\lambda(\gamma)+\frac{1}{\gamma}\ln\sum_{y\in E}e^{\gamma w(y,\gamma)}\bP^a(x,y)\right],& \textrm{if } \gamma\neq 0,\\
 w(x,\gamma)=\max_{a\in U}\left[c(x,a)-\lambda(\gamma)+\sum_{y\in E} w(y,\gamma)\bP^a(x,y)\right],& \textrm{if }\gamma=0.
\end{cases}
\end{equation}
Under the assumptions stated in Section~\ref{S:ergodic}, a solution to the Bellman equation always exists and could be used to recover the optimal Markov strategies for the initial problem \eqref{eq:RSC}.  This result could be recovered, e.g., by using the span-contraction approach that is using the contractive property of the Bellman operator
\begin{equation}\label{eq:Bellman.op}
T_{\gamma}g(x):=
\begin{cases}
\max_{a\in U}\left[c(x,a)+\frac{1}{\gamma}\ln\sum_{y\in E}e^{\gamma g(y)}\bP^a(x,y) \right],& \gamma\neq 0,\\
\max_{a\in U}\left[c(x,a)+\sum_{y\in E}g(y)\bP^a(x,y) \right],& \gamma=0,
\end{cases}
\end{equation}
for $g\colon E\to \bR$, under the span norm, that it, a semi-norm defined as
\[
\textstyle \|g\|_{\textnormal{sp}}:=\tfrac{1}{2}[\sup_{x\in E}g(x)-\inf_{x\in E}g(x)],\quad \textrm{for } g\colon E\to \bR.
\]

\begin{theorem}\label{th:local.contraction}
Assume \eqref{A.1}. Then, for any $\gamma\in\bR$, the Bellman operator defined in \eqref{eq:Bellman.op} is a local contraction in the span norm, i.e. there exists an (increasing) function $\Lambda\colon \bR_{+}\to (0,1)$ such that for any $g_1,g_2\colon E\to \bR$ we have
\[
\|T_{\gamma}g_1 -T_{\gamma}g_1\|_{\textnormal{sp}}=\Lambda(M)\|g_1-g_2\|_{\textnormal{sp}},
\]
where $M:=\|g_1\|_{\textnormal{sp}} \vee\|g_1\|_{\textnormal{sp}}$ is the local span-norm induced constant. Furthermore, assuming additionally \eqref{A.2}, we have $\sup_{n\in\bN}\|T^n_{\gamma}(0)\|_{\textnormal{sp}}<\infty$, i.e. the consequent iterations of the operator (starting from 0) are jointly bounded in the span-norm.
\end{theorem}
For brevity, we omit the proof of Theorem~\ref{th:local.contraction} and details of the consequent constructions, see e.g. \cite{DiMSte1999}. Note that, by the joint boundedness of the iterated sequence $(T^n_{\gamma}(0))_{n\in\bN}$, we get the global contraction property on the appropriate subspace, so that we can use Banach's fixed-point theorem to recover the unique fixed point (defined up to an additive constants) of the Bellman operator~\eqref{eq:Bellman.op}. This proves the existence of the solution of~\eqref{eq:Bellman1} as stated in the next result.

\begin{theorem}\label{th:exist}
Assume \eqref{A.1} and \eqref{A.2}. Then, for any $\gamma\in\bR$, there exists a unique solution $(w(\cdot,\gamma),\lambda(\gamma))$ to the Bellman equation~\eqref{eq:Bellman1}, where $w(\cdot,\gamma)$ is defined up to an additive constant.  Moreover, the stationary Markov policy $\hat u\in\Pi'$ given by
\begin{equation}\label{eq:Markov.stat.strategy}
\hat u(x)=
\begin{cases}
\argmax_{a\in U}\left[c(x,a)-\lambda(\gamma) + \frac{1}{\gamma}\ln \left( \sum_{y\in E} e^{w(y,\gamma)}\mathbb{P}^a(x,y)\right)\right],& \gamma\neq 0,\\
\argmax_{a\in U}\left[c(x,a)-\lambda(\gamma) + \sum_{y\in E} w(y,\gamma)\mathbb{P}^a(x,y)\right],& \gamma=0,
\end{cases}
\end{equation}
is optimal for \eqref{eq:RSC.opt}, that is, $J_{\gamma}(x,\hat u)=J^*_\gamma(x)=\lambda(\gamma)$, for $x\in E$.
\end{theorem}

Since the proof of Theorem~\ref{th:exist} is a direct consequence of Theorem~\ref{th:local.contraction}, we omit its proof. We note that, under \eqref{A.1}, the existence of a solution to Bellman's equation~\eqref{eq:Bellman1} for $\gamma=0$ also follows directly from Proposition 1 in \cite{Ste2023}; assumption \eqref{A.2} is effectively not required in the risk-neutral case. For $\gamma\neq 0$, an alternative proof idea using the Krein-Rutman theorem can be found in \cite{Ste2024}; this technique, in contrast to the span-contraction approach, does not directly ensure solution's uniqueness (up to an additive constant). Also, the solution existence statement could be recovered by analysing the related eigenvalue problems and applying the Perron-Frobenius Theorem (which is in fact a special case of Krein-Rutman theorem); see proof of Theorem~\ref{sec:riskaversion_parameter} for details. Finally, note that the Markov strategy \eqref{eq:Markov.stat.strategy} is induced by Bellman's equation and constructed by following the standard iterative procedure, see e.g. \cite{PitSte2024}.

\subsection{Entropic representation and duality}
The objective criterion introduced in \eqref{eq:RSC} is directly related to the entropic utility. For a generic probability measure $\bP$ and related expectation $\bE$, the entropic utility is given by
\begin{equation}\label{eq:entropy}
\ent(Z,\gamma):=
\begin{cases}
\frac{1}{\gamma}\ln \bE\left[e^{\gamma Z} \right] &\textrm{if } \gamma\neq 0\\
\bE[X] &\textrm{if } \gamma =0
\end{cases}
, \quad Z\in L^{\infty}(\Omega,\cF,\bP).
\end{equation}
Indeed, we can rewrite \eqref{eq:RSC} as
\[
\textstyle J_\gamma(x,\pi)=\liminf_{n\to\infty}\tfrac{1}{n}\ent_x^{\pi}\left(\sum_{i=0}^{n-1}c(X_i,a_i),\gamma\right)
\]
where $\ent_x^{\pi}$ is the entropic utility operator for $\bP_x^{\pi}$. A Taylor series expansion around $\gamma=0$ shows that
\begin{equation}
 \ent(Z,\gamma) \approx   \bE[Z] +\tfrac12 \gamma \textnormal{Var}(Z)
\end{equation}
which yields the stated risk-sensitivity parameter interpretation and links risk-sensitive criterion to the classical mean-variance objective framework, see \cite{BiePli2003} for more details.
The mapping \eqref{eq:entropy} satisfies a number of (unique) properties, which makes it useful in long-run stochastic control. In particular, it admits a robust (dual) representation of the form
\begin{equation}\label{eq:entropy.dual}
\ent(Z,\gamma)=\inf_{\bQ\in \cP(\Omega)}\left\{\bE_{\bQ}[Z] -\tfrac{1}{\gamma}H[\bQ\,\|\, \bP] \right\},
\end{equation}
where $\cP(\Omega)$ denotes the set of all probability measures on $(\Omega,\cF)$ and $H[\bQ\,\| \, \bP]$ is the relative entropy (Kullback-Leibler divergence) of $\bQ$ w.r.t.\  $\bP$ given by
\begin{equation}\label{eq:relative.entropy}
H[\bQ\,\|\, \bP]:=
\begin{cases}
\bE_{\bQ}\left[\ln\frac{\dif\bQ}{\dif\bP}\right] & \textrm{if } \bQ\ll \bP,\\
+\infty &\textrm{otherwise}.
\end{cases}
\end{equation}
Furthermore, entropic utility is in fact a convex risk measure (for $\gamma<0$), certainty equivalent, as well as mean value principle. Throughout the paper, we often use the respective monotonicity, additivity, or time-consistency properties, we refer to \cite{KupSch2009} for details. 

\section{Stationary Markov policies and their risk-aversion stability}\label{sec:riskaversion_parameter}
This section mainly aims to explore the stability and continuity of stationary Markov policies and optimal value functions in relation to the risk aversion parameter $\gamma \in \mathbb{R}$. As expected, small perturbations in the risk aversion specification typically do not result in optimal policy change. For brevity, for any $\gamma\in\bR$, we define the set of (attainable) stationary Markov optimal strategies given by
\begin{equation}\label{eq:Pi*}
\Pi^{*}_{\gamma}:=\{u\in\Pi'\colon  J_{\gamma}(x,u)=J^*_{\gamma}(x),\, x\in E\},
\end{equation}
Also, for $u\in\Pi'$ and $\beta\in (0,1]$ we define the set
\begin{equation}\label{eq:Gamma}
\Gamma(u):=\{\gamma\in\bR\colon u\in \Pi^{*}_{\gamma}\},
\end{equation}
which identifies a set of risk aversion parameters for which the Markov policy $u\in\Pi'$ is optimal for any starting point $x\in E$. 

From Theorem~\ref{th:exist} we know that under \eqref{A.1} and \eqref{A.2} the set $\Pi^{*}_{\gamma}$ is nonempty for any $\gamma\in\bR$. To find the optimal policy, instead of solving the Bellman equation \eqref{eq:Bellman1} directly, we can consider a (finite) set of Poisson equations linked to fixed Markov policies and then maximize the target value. Namely, for $\gamma\in \bR$ and $u\in \Pi'$ we consider equation
\begin{equation}\label{eq:MPE}
\begin{cases}
\textstyle w^u(x,\gamma)=c(x,u(x))-\lambda^u(\gamma)+\frac{1}{\gamma}\ln\sum_{y\in E}e^{\gamma w^u(y,\gamma)}\bP^u(x,y),& \textrm{if } \gamma\neq 0,\\
 w^u(x,\gamma)=c(x,u(x))-\lambda^u(\gamma)+\sum_{y\in E} w^u(y,\gamma)\bP^u(x,y),& \textrm{if }\gamma=0,
\end{cases}
\end{equation}
and refer to it as the {\it Multiplicative Poisson Equation} (MPE), if $\gamma\neq 0$, and {\it Additive Poisson Equation} (APE), if $\gamma=0$. As before, we call a pair $(w^u(\cdot,\gamma),\lambda^u(\gamma))$, where $w^u(\cdot,\gamma)\colon E\to\bR$ and $\lambda^u(\gamma)\in \bR$, a solution to MPE/APE if the pair solves \eqref{eq:MPE}.
Let us now summarize the properties of MPE and APE solutions from \eqref{eq:MPE} and their link to the original problem.

\begin{theorem}\label{th:MPE.closed}
Assume \eqref{A.1} and \eqref{A.2}. Then,
\begin{enumerate}
\item For any $\gamma\in\bR$ and $u\in\Pi'$, there is a solution to MPE/APE equation~\eqref{eq:MPE};
\item For any $u\in \Pi'$, the function $\gamma\to \lambda^u(\gamma)$ is continuous, bounded, and non-decreasing.  Furthermore, $\gamma\to \lambda^u(\gamma)$ is real analytic.
\item For any $\gamma\in \bR$, we have $\max_{u\in\Pi'}\lambda^u(\gamma)=\lambda(\gamma)$, where $\lambda(\gamma)$ is from \eqref{eq:Bellman1}.
\item For any $\gamma\in\bR$ and $u\in \arg\max_{u\in\Pi'}\lambda^u(\gamma)$, we have $u\in\Pi_{\gamma}^*$.
\item  For any $u\in\Pi'$ the set $\Gamma(u)$ is closed and can be represented as a countable union of closed intervals.

\item For any $u\in \Pi'$ we have $\lambda^u(\gamma)\nearrow \lambda^u(\infty)$, for some $\lambda^u(\infty)\in\bR$, as $\gamma\to\infty$. Moreover, if $\lambda(\infty):=\max_{u\in \Pi'}\lambda^{u}(\infty)$ is realised for a unique $u^*\in\Pi'$, then there exists $\gamma_\infty$  such that $\Pi_\gamma^*=\{u^{*}\}$, for  $\gamma\ge \gamma_\infty$. A similar statement holds for negative $\gamma$.
\end{enumerate}
\end{theorem}

\begin{proof} For brevity, we focus on the case $\gamma\neq 0$ and MPE equations.

\medskip

\noindent (1.) The proof of the existence of the MPE solution follows directly from~\cite{CavHer2009}. This is due to the fact that \eqref{A.2} together with the aperiodicity induced by \eqref{A.1} implies the uni-chain condition, and existence of a (unique) recurrent class to which we enter in the finite number of steps; see also \cite{PitSte2024} where the existence of an MPE solution is discussed in the mixing framework. For completeness, we note that the existence of MPE solution in the finite setting is, in fact, a direct consequence of the Perron-Frobenius Theorem. Indeed, for $\gamma\neq 0$ we can convert \eqref{eq:MPE} to the eigenvalue problem $Av=rv$, where
\begin{equation}\label{eq:MPE.matrix.repr}
A :=\left[\bP^u(x_i,x_j)e^{\gamma c(x_i,u(x_i))}\right]_{i,j=1}^{k},\quad
v := [e^{\gamma w^{u}(x_i,\gamma)}]_{i=1}^{k},\quad
r := e^{\gamma\lambda^{u}(\gamma)}.
\end{equation}
In this setting, given matrix $A$, we are looking for a strictly positive eigenvalue $r$ and eigenvector $v$. Without loss of generality, let us assume that $A$ is irreducible; otherwise, one needs to restrict $A$ to states from the unique recurrent class. The Perron-Frobenius Theorem states that the largest eigenvalue of $A$ is positive and there exists an eigenvector with positive entries, see \cite{Mey2000}. This shows that the solution to MPE exists.

\medskip 

\noindent (2.) To show monotonicity of the function $\gamma\to\lambda^u(\gamma)$, we first note that for any $u\in\Pi'$, the value $\lambda^u(\gamma)$ is the value of the long-run average risk sensitive criterion for policy $u$, that is, we have
\begin{equation}\label{eq:MPE.eq1}
\lambda^{u}(\gamma)=\lim_{n\to\infty}\tfrac{1}{n}\tfrac{1}{\gamma }\ln\bE_x^{u}\left[e^{\gamma\sum_{i=0}^{n-1}c(X_i,u(X_i))}\right],\quad x\in E,
\end{equation}
see \cite{PitSte2024} for details.
Thus, using \eqref{eq:entropy}, we can rewrite \eqref{eq:MPE.eq1} as
\begin{equation}\label{eq:MPE.entropic}
\lambda^u(\gamma)=\lim_{n\to\infty}\frac{1}{n}\textrm{Ent}_{x}^u\left(\sum_{i=0}^{n-1}c(X_i,u(X_i)),\gamma\right).
\end{equation}
Since the entropic utility is monotone with respect to $\gamma$, the monotonicity of $\gamma\to\lambda^u(\gamma)$ follows directly. From \eqref{eq:MPE.entropic}, we also get the boundedness of the function $\gamma\to\lambda^u(\gamma)$ since the entropic utility can be bounded from below (resp. above) by the essential infimum (resp. supremum) of the underlying random variable, which implies
\[
-\|c\| \leq \lambda^u(\gamma)\leq \|c\|,\quad \gamma\in\bR,
\]
where $\|\cdot\|$ denotes the supremum norm, see \cite{KupSch2009} for details. Next, let us show the continuity of $\gamma\to\lambda^u(\gamma)$ and later show that this function is in fact also real analytic.

To show the continuity of $\gamma\to\gamma\lambda^u(\gamma)$, we first note that the function $\gamma\to\gamma\lambda^u(\gamma)$ is convex. This follows directly from \cite{Kin1961}, where it has been shown that if all the entries in the matrix $A$ in the representation \eqref{eq:MPE.matrix.repr} are a logconvex function of the underlying parameter $\gamma$, then the largest eigenvalue must also be logconvex, which implies convexity of $\gamma\to\gamma\lambda^u(\gamma)$. Consequently, the function $\gamma\to\gamma\lambda^u(\gamma)$ is continuous on $\bR$, which implies continuity of $\gamma\to\lambda^u(\gamma)$ on $\bR\setminus\{0\}$.  For brevity, we skip the continuity proof of $\gamma\to\lambda^u(\gamma)$ for $\gamma=0$; see Equation~(36) in \cite{Ste2023} or \cite{chavez2015continuity} for the proof idea.

Let us finally show the real analytic property of $\gamma\to\lambda^u(\gamma)$. From representation \eqref{eq:MPE.matrix.repr}, the Perron-Frobenius Theorem and Theorem 8 in Chapter 9 of \cite{Lax2007}, where it is shown that if the matrix $A$ has analytic entries (w.r.t.\ $\gamma$), then the simple eigenvalues are also analytic (w.r.t.\ $\gamma$), we get that $\gamma\to e^{\gamma\lambda^u(\gamma)}$ is analytic on the real line, and so is $\gamma\to \gamma\lambda^u(\gamma)$.  This implies that on $(-\varepsilon,\varepsilon)$ with $\varepsilon>0$ there is a Taylor series of $\gamma\lambda^u(\gamma)=\sum_{k=0}^\infty a_k \gamma^k$ around $\gamma=0$. Moreover, we find that $\gamma\to \lambda^u(\gamma)$ is analytic for $\gamma >0$ and for $\gamma <0$. Furthermore, since $\gamma\to\lambda^u(\gamma)$ is continuous on $\bR$, we obtain a Taylor series for $\lambda^u(\gamma)$ around zero (i.e. $a_0=0$). Considered together, this implies that $\gamma\to \lambda^u(\gamma)$ is real analytic on $\bR$.

\medskip

\noindent (3.--4.) For the next part of the proof, it is sufficient to recall that for any $u\in\Pi'$ the value $\lambda^u(\gamma)$ is recovering the long-run averaged risk sensitive criterion value for control $u$ as stated in \eqref{eq:MPE.eq1}. Consequently, the maximal value is attained and recovered by one of the Markov polices; such policy must exist due to Theorem~\ref{th:exist}.

\medskip

\noindent (5.) Fix $u\in\Pi'$. From (2.), we see that for any $\tilde u\in \Pi'$, the function $F_{u,\tilde u}\colon \bR\to\bR$ given by $F_{u,\tilde u}(\gamma):=\lambda^u(\gamma)-\lambda^{\tilde u}(\gamma)$ is continuous. Consequently, its upper zero level set $A(u,\tilde u):=\{\gamma\in\bR: \lambda^u(\gamma)\geq \lambda^{\tilde u}(\gamma)\}$ is closed. This implies closedness of $\Gamma(u)$ as
\[
\textstyle \Gamma(u)=\bigcap_{\tilde u\in\Pi'} A(u,\tilde u).
\]
It remains to show that $\Gamma(u)$ is a countable union of closed intervals. Since $|\Pi'|<\infty$, it is sufficient to show that for an arbitrary $\tilde u\in\Pi'$, the set $A(u,\tilde u)$ is a countable union of closed intervals. For simplicity, let us focus on the case where $\gamma\leq 0$; similar reasoning is valid for $\gamma\geq 0$. From the discussion in (2.) we get that the functions $\gamma\to\lambda^u(\gamma)$ and $\gamma\to\lambda^{\tilde u}(\gamma)$ are real analytic on $\bR$. Consequently, the zeros of the function $F_{u,\tilde u}$ are isolated in $\bR$ or $F_{u,\tilde u}\equiv 0$ on $\bR$. As $F_{u,\tilde u}$ is continuous, this directly implies that the set $A(u,\tilde u) $ is the countable union of intervals.
%; note that the zero accumulation point in $\gamma=0$ is allowed.
%%%%%%%

  \noindent (6.) Fix $u\in \Pi'$. We know already from (2.) that $\gamma\to \lambda^u(\gamma)$ is  bounded and non-decreasing. Thus, a limit $\lambda^u(\gamma)\to \lambda^u(\infty)$, as $\gamma\to\infty$, exists. Let us now assume that there is a unique $u^* \in \Pi'$ such that $\lambda^{u^*}(\infty)=\lambda(\infty)$. Again since for any $u\in \Pi'$ the function $\gamma\to \lambda^u(\gamma)$ is non-decreasing and converges to its upper value $\lambda^u(\gamma)$, there exists $\gamma_{\infty} >0$ such that
\[
\textstyle \lambda^{u*}(\gamma_{\infty})>\sup_{u\in\Pi'\setminus\{u^*\}}\lambda^u(\infty).
\]
This indicates that for any $\gamma>\gamma_{\infty}$ we get $\Pi_\gamma^*(1)=\{u^{*}\}$, which concludes the proof.
\end{proof}
In the next corollary, we present some interesting properties of optimal strategies in the finite setting that could be inferred from Theorem~\ref{th:MPE.closed} combined with Theorem~\ref{th:exist}.

\begin{corollary}\label{cor:1}
Assume \eqref{A.1} and \eqref{A.2}.  Then:
\begin{enumerate}
\item The family of sets $\{\Gamma(u)\}_{u\in\Pi'}$ is a closed finite cover of the risk-sensitivity parameter space, that is, $\bR=\bigcup_{u\in\Pi'}\Gamma (u)$ and for any $u\in\Pi'$ the set $\Gamma(u)$ is closed.
\item If there is a unique optimal Markov policy for $\gamma\in\bR$, that is, if $\Pi_{\gamma}^*=\{u\}$ for some $u\in\Pi'$, then there exists $\epsilon>0$ such that $[\gamma-\epsilon,\gamma+\epsilon] \subset \Gamma(u)$.
\item If for any $\gamma\in\bR$ there is a unique Markov policy, then there exists a globally optimal policy, i.e. $u\in\Pi'$ such that $\Gamma(u)=\bR$.
\item If $\gamma_n\to\gamma$ and $u_n\in\Pi'$ is optimal for $\gamma_n$, $n\in\bN$, then $\Pi^*_\gamma$ contains all accumulation points of $(u_n)$.
\end{enumerate}
\end{corollary}
\begin{proof}
The proof of 1. is a straightforward implication of Theorem~\ref{th:MPE.closed} combined with Theorem~\ref{th:exist}. Proof of 2. is by contradiction. Suppose that the statement is not correct. Then there exists a sequence $(\gamma_n)$ with $\gamma_n\to\gamma$ for $n\to\infty$ such that $u\notin \Pi_{\gamma_n}^*$. Let $u_n\in \Pi_{\gamma_n}^*$. Since $\Pi'$ is finite, there is a subsequence $m_n$ and a policy $\tilde u\neq u$ such that $u_{n_m}=\tilde u.$ Then, by Theorem \ref{th:MPE.closed}, policy $\tilde u$ also has to be optimal for $\gamma$, which leads to contradiction. The proof of 3) and 4) follows from 2) and Theorem~\ref{th:MPE.closed}.
\end{proof}

 \begin{remark}[Typical structure of the $\Gamma$ optimality sets]
 From Corollary~\ref{cor:1}, we can infer a typical structure of the set $\Gamma$, which is presented in Figure~\ref{F:Gamma}. Different colors belong to different policies, that is, $\Gamma(u_1)$ is blue, $\Gamma(u_2)$ is red, etc. On top of this, there can be a finite number of points where additional policies are optimal. At the endpoint of intervals, we have at least two different optimal policies. 
 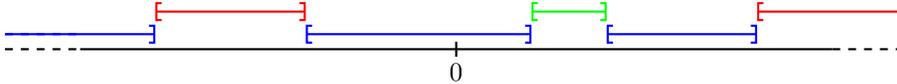
\begin{figure}[h]
    \centering
  \begin{tikzpicture}
   \draw[dashed, thick] (-6,-0.2) -- (-5,-0.2);
   \draw[thick] (-5,-0.2) -- (5,-0.2);
     \draw[dashed, thick] (5,-0.2) -- (6,-0.2);
 \draw[ {-]} , thick,blue] (-6,0) -- (-4,0);
   \draw[dashed, thick,blue] (-6,0) -- (-5,0);
  \draw[ {[-]}, thick,red] (-4,0.3) -- (-2,0.3);
 \draw[ {[-]} , thick,blue] (-2,0) -- (1,0);
   \draw[ {[-]}, thick,green] (1,0.3) -- (2,0.3);
    \draw[ {[-]} , thick,blue] (2,0) -- (4,0);
      \draw[ {[-}, thick,red] (4,0.3) -- (6,0.3);
    \draw[thick] (0,-0.3) -- (0,-0.1);  
     \draw (0,-0.5) node {0};
  \end{tikzpicture}
    \caption{Typical decomposition of the real line in sets $\Gamma(u)$.}\label{F:Gamma}
\end{figure}
\end{remark}

\begin{remark}[Link between risk-neutral and risk-averse optimal policies]
Theorem~\ref{th:MPE.closed} indicates that there exists a risk-neutral optimal policy which is also optimal for the risk-averse problem when $\gamma<0$ is sufficiently close to zero. Similarly, there exists a risk-neutral optimal policy, which is also optimal for the risk-seeking problem when $\gamma>0$ is sufficiently close to zero. That being said, the set of such policies might be disjoint as illustrated in Example~\ref{ex:1}. On the other hand, Corollary \ref{cor:1} implies that if there is a unique risk-neutral policy, then this policy is also optimal for small $\gamma>0$ and small $\gamma<0$.
\end{remark}

\section{Vanishing discount approach and Blackwell stability of risk-sensitive optimal policies}\label{sec:blackwell}

In this section, we want to investigate whether one can solve the long-run averaged risk sensitive stochastic control problem by considering a series of discounted risk sensitive stochastic control problems with increasing discount factor (up to 1). In particular, we are interested in the {\it Blackwell property} that ensures the averaged optimality of a given optimal discounted-type policy above a certain discount factor threshold.

Before we present the key results of this section, we need to introduce additional notation related to the discounted risk-sensitive framework. We start with the discounted equivalent of the objective function provided in \eqref{eq:RSC}. For a pre-specified risk-averse parameter $\gamma\in\bR$ as well as discount factor $\beta\in (0,1)$ the  {\it discounted long run risk-sensitive criterion} is given by
\begin{equation}\label{eq:RSC.disc}
J_\gamma(x,\pi; \beta):=
\begin{cases}
\tfrac{1}{\gamma}\ln\bE_x^{\pi}\left[e^{\gamma\sum_{i=0}^{\infty}\beta^i c(X_i,a_i)}\right], & \textrm{if }\gamma\neq 0,\\
\bE_x^{\pi}\left[\sum_{i=0}^{\infty}\beta^i c(X_i,a_i)\right], & \textrm{if } \gamma=0.
\end{cases}
\end{equation}
For brevity, and with a slight abuse of notation, we also set $J_{\gamma}(x,\pi;1):=J_{\gamma}(x,\pi)$. Also, note that \eqref{eq:RSC.disc} could be represented with the entropic utility, i.e. we get
\[
\textstyle J_\gamma(x,\pi;\beta)=\ent_x^{\pi}\left(\sum_{i=0}^{\infty}\beta^i c(X_i,a_i),\gamma\right).
\]
Next, we can also introduce an equivalent of the Bellman equation~\eqref{eq:Bellman1} considering
\begin{equation}\label{eq:Bellman2}
\begin{cases}
\textstyle w^{\beta}(x, \gamma)=\max_{a\in U}\left[\gamma c(x,a)+\ln\sum_{y\in E}e^{w^{\beta}(y,\beta \gamma)}\bP^a(x,y)\right],& \textrm{if } \gamma> 0,\\
\textstyle w^{\beta}(x, \gamma)=\min_{a\in U}\left[\gamma c(x,a)+\ln\sum_{y\in E}e^{w^{\beta}(y,\beta \gamma)}\bP^a(x,y)\right],& \textrm{if } \gamma< 0,\\
 w^{\beta}(x,\gamma)=\max_{a\in U}\left[c(x,a)+\beta\sum_{y\in E} w^{\beta}(y,\gamma)\bP^a(x,y)\right],& \textrm{if }\gamma=0,
\end{cases}
\end{equation}
for function $w^{\beta}: E\times \bR \to \bR$; for technical reasons, for $\gamma\neq 0$, we renormalized the value function by $\gamma^{-1}$ in \eqref{eq:Bellman2} and split the formula into the risk-seeking, risk-averse, and risk-neutral case. One can show that in the discounted case the optimal value of the initial problem can be directly recovered from \eqref{eq:Bellman2}, that is, we get
\[
J^*_\gamma(x;\beta):=\sup_{\pi\in\Pi}J_\gamma(x,\pi; \beta)=w^{\beta}(x,\gamma)/\gamma,
\]
and there is no need to introduce the long-run averaging value $\lambda$; for brevity, we omit the detailed proof of this fact. On the other hand, we note that while in \eqref{eq:Bellman1} the function $w$ was effectively defined for a single isolated (prefixed) $\gamma\in\bR$, this is not the case for \eqref{eq:Bellman2} (and $\gamma\neq 0)$, where it is necessary to consider $w^{\beta}$ on the risk-aversion grid $(\beta^n\gamma)_{n\in\bN}$ in order to allow the recursive scheme. Given the initial parameter $\gamma >0$ (resp. $\gamma<0$) the function $w^{\beta}$ in \eqref{eq:Bellman2} can be defined on $E\times (0,\gamma]$ (resp. $E\times [\gamma,0)$). In particular, this indicates that the chosen optimal strategy might vary between iteration steps and the recursive procedure might lead to the emergence of a non-stationary optimal policy, see Example~\ref{ex:non-stat} where we show this might indeed be the case. For consistency, we also expand the notation introduced in \eqref{eq:Pi*}-\eqref{eq:Gamma} and for any $\beta\in (0,1)$, $\gamma\in\bR$ and $u\in\Pi$, introduce sets
\[
\Pi^{*}_{\gamma}(\beta):=\{\pi\in\Pi''\colon  J_{\gamma}(x,\pi;\beta)=J^*_{\gamma}(x;\beta),\, x\in E\},\quad 
\Gamma(\pi;\beta):=\{\gamma\in\bR\colon \pi\in \Pi^{*}_{\gamma}(\beta)\};
\]
note that in contrast to the previous section, we do not restrict the definitions to the stationary Markov policies. 

The central point of the vanishing discount approach is to relate the value of $J^*(x;\beta)$ to $J^*(x)$, when $\beta\to 1$, as well as to link the respective optimal policies. While in the risk-neutral case the link could typically be expressed using a simple relationship
\begin{equation}\label{eq:vanishing.intro}
\lim_{\beta\uparrow 1}(1-\beta)J_{0}^{*}(x,\beta)=J_{0}^*(x),
\end{equation}
this is no longer the case for risk-sensitive case, as one needs to take into account potential non-stationarity of the induced strategies, non-additive nature of the underlying risk-sensitive functional, impact of the span-norm usage, etc. We refer to \cite{CavFer2000, CavHer2017} for variations of formula \eqref{eq:vanishing.intro} in the risk-sensitive framework and for more details on this subject. Still, one expects that the policies optimal for the discounted problem for $\beta$s close to 1 stabilize and become optimal for the averaged problem. This concept is linked to the Blackwell property in the risk-neutral setup. For completeness, let us recall this classical result, we refer to \cite{Bla1962,DewGal2022,GraPet2023} for a more comprehensive discussion about the related concept of Blackwell optimality.

\begin{theorem}\label{th:blackwell}
Assume \eqref{A.1} and fix $\gamma=0$.  Then, there exists a stationary Markov policy $\pi\in\Pi'$ which satisfies the Blackwell property, i.e. there is $\beta_0\in (0,1)$ such that for any $\beta\in (\beta_0,1)$ we have $\pi\in \Pi^*_0(\beta)$. Furthermore, we have $\pi \in \Pi^*_{0}$.
\end{theorem}

For brevity, we omit the proof of Theorem~\ref{th:blackwell}, see e.g. the proof of Theorem 10.1.4 and Corollary 8.2.5 in \cite{Put2009}. Unfortunately, Theorem~\ref{th:blackwell} and its proof cannot be directly extended to the risk-sensitive framework due to two main reasons: (1) the classical proof is based on the fact that optimal solutions could be represented as rational functions (w.r.t. $\beta$) which does not transfer to the multiplicative framework; (2) in contrast to the risk-neutral setting, the set of optimal policies for the risk-sensitive discounted problem could only contain non-stationary policies.

\subsection{Using discounted framework to approximate optimal strategies: risk-sensitive vanishing discount approach}
In this section we state the selected properties of optimal discounted problem values and policies when $\beta\to 1$. If not stated otherwise, in this section we focus on the case $\gamma<0$; the proofs and formulas for $\gamma>0$ are similar. We start with a simple proposition that puts a uniform bound on the value function.
\begin{proposition}\label{pr:w.disc.bound}
Assume \eqref{A.1} and \eqref{A.2}; fix $\gamma\neq 0$. Then, for any $\beta\in (0,1)$ we have
\[
\|w^{\beta}(\cdot,\gamma)\|_{\textnormal{sp}}\leq |\gamma|N\|c\|+\tfrac{1}{2}\ln K,
\]
where $N\in\bN$ and $K<\infty$ are constants stated in \eqref{A.2}.
\end{proposition}
\begin{proof}
Fix $\beta<1$ and $\gamma<0$; the proof for $\gamma>0$ is similar. Let $N\in\bN$ and $K<\infty$ be as in \eqref{A.2}. Iterating the Bellman equation \eqref{eq:Bellman2} $N$ times, we get
\begin{equation}\label{eq:bellman.n.times}
w^{\beta}(x,\gamma)=\inf_{\pi\in \Pi''}\ln\bE^{\pi}_x\left[e^{\gamma\sum_{i=0}^{N-1}\beta^i c(X_i,a_i)+w^{\beta}(X_N,\gamma\beta^N)}\right].
\end{equation}
Thus, bounding the sum in \eqref{eq:bellman.n.times} by $\pm N\|c\|$, we get
\begin{equation}\label{eq:disc.bound.1}
|w^{\beta}(x, \gamma)- w^{\beta}(y, \gamma)|  \leq 2|\gamma|N\|c\|+\sup_{\pi\in \Pi''}\ln\frac{\bE^{\pi}_x[e^{w^{\beta}(X_N,\gamma\beta^N)}]}{\bE^{\pi}_y[e^{w^{\beta}(X_N,\gamma\beta^N)}]},
\end{equation}
for any $x,y\in E$. Also, using \eqref{A.2}, for any $x,y\in E$ and $\pi\in \Pi''$, we have
\begin{equation}\label{eq:disc.bound.2}
\frac{\bE^{\pi}_x[e^{w^{\beta}(X_N,\gamma\beta^N)}]}{\bE^{\pi}_y[e^{w^{\beta}(X_N,\gamma\beta^N)}]} = K+\frac{\sum_{z\in E}e^{w^{\beta}(z,\gamma\beta^N)}[\bP^{(\pi,N)}(x,z)-K\bP^{(\pi,N)}(y,z)]}{\sum_{z\in E}e^{w^{\beta}(z,\gamma\beta^N)}\bP^{(\pi,N)}(y,z)}\leq K.
\end{equation}
Combining \eqref{eq:disc.bound.1} and \eqref{eq:disc.bound.2}, we conclude the proof
\end{proof}
Next, we follow the standard vanishing discount framework in which we center the value function to link span-norm into standard supremum norm and define the iterated incremental value. Namely, let us fix $\bar z \in E$ and set
\begin{align}
\bar w^{\beta}(x,\gamma) &:=w^{\beta}(x,\gamma)-w^{\beta}(\bar z,\gamma),\label{eq:vanishing1}\\
\lambda_n^{\beta}(\gamma) &:= w^{\beta}(\bar z,\gamma\beta^n)- w^{\beta}(\bar z,\gamma\beta^{n+1}),\label{eq:vanishing2}\\
\bar w_n^{\beta}(x,\gamma) &:=\bar w^{\beta}(x,\gamma\beta^n),\label{eq:vanishing3}
\end{align}
for $n\in\bN$, $\beta\in (0,1)$, $\gamma<0$, and $x\in E$. For $\gamma<0$, we can use \eqref{eq:vanishing1}--\eqref{eq:vanishing3} to reformulate \eqref{eq:Bellman2} as
\begin{equation}\label{eq:bellman.vanishing1}
\textstyle \bar w^{\beta}(x, \gamma)=\min_{a\in U}\left[\gamma c(x,a)-\lambda_0^{\beta}(\gamma)+\ln\sum_{y\in E}e^{\bar w^{\beta}(y,\beta \gamma)}\bP^a(x,y)\right].
\end{equation}
Furthermore, for $\gamma<0$ and $n\in\bN$, we have
\begin{equation}\label{eq:bellman.vanishing2}
\textstyle \bar w_n^{\beta}(x, \gamma)=\min_{a\in U}\left[\gamma\beta^n c(x,a)-\lambda_n^{\beta}(\gamma)+\ln\sum_{y\in E}e^{\bar w_{n+1}^{\beta}(y,\gamma)}\bP^a(x,y)\right].
\end{equation}
Now, let us show that the values $(\lambda_n^{\beta}(\gamma))$ are uniformly bounded w.r.t.\ \ $\beta$ and $n\in\bN$.

\begin{proposition}\label{pr:lambda.n.bound}
Assume \eqref{A.1} and \eqref{A.2}; fix $\gamma\neq 0$. Then, we have
\[
\sup_{\beta\in (0,1)}\sup_{n\in\bN}|\lambda_n^{\beta}(\gamma)|\leq |\gamma|\|c\|(1+2N)+\ln K,
\]
where $N\in\bN$ and $K<\infty$ are the constants stated in \eqref{A.2}.
\end{proposition}

\begin{proof}
Fix $\gamma<0$; the proof for $\gamma>0$ is similar. Using Proposition~\ref{pr:w.disc.bound} and \eqref{eq:Bellman2}, for any $\beta\in (0,1)$, $n\in\bN, x\in E$, and (optimal) $a\in U$, we have
\begin{align*}
w^{\beta}(x, \gamma\beta^n)& \leq |\gamma|\beta^n \|c\|+\ln\sum_{y\in E}e^{w^{\beta}(y,\beta^{n+1} \gamma)}\bP^a(x,y)\\
& \leq w^{\beta}(x,\beta^{n+1} \gamma)+|\gamma|\beta^n \|c\|+\ln\sum_{y\in E}e^{|w^{\beta}(y,\beta^{n+1} \gamma)-w^{\beta}(x,\beta^{n+1} \gamma)|}\bP^a(x,y)\\
& \leq w^{\beta}(x,\beta^{n+1} \gamma)+|\gamma|\beta^n \|c\|+2|\gamma\beta^{n+1}|N\|c\|+\ln K\\
&\leq w^{\beta}(x,\beta^{n+1} \gamma)+|\gamma|\|c\|(1+2N)+\ln K. 
\end{align*}
Consequently, we get $\lambda_n^{\beta}(\gamma)\leq|\gamma|\|c\|(1+2N)+\ln K$, for any $n\in\bN$ and $\beta\in (0,1)$. Using similar reasoning, for $n\in\bN$ and $\beta\in (0,1)$, we get
\[
w^{\beta}(x, \gamma\beta^n) \geq w^{\beta}(x,\beta^{n+1} \gamma)-|\gamma|\|c\|(1+2N)-\ln K, 
\]
which implies $\lambda_n^{\beta}(\gamma)\geq -|\gamma|\|c\|(1+2N)-\ln K$, and concludes the proof.
\end{proof}
Next, we are ready to state the vanishing discount convergence result which is the key result of this section.

\begin{theorem}\label{th:vanishing.discount}
Assume \eqref{A.1} and \eqref{A.2}; fix $\gamma\neq 0$. Then, there exists a function $ w(\cdot,\gamma)\colon E\to \bR$ and a constant $\lambda(\gamma)\in\bR$, such that, for $n\in\bN$, we have
\[
\lim_{\beta\uparrow1} \bar w_n^{\beta}(x,\gamma)= w(x,\gamma)\quad \textrm{and}\quad \lim_{\beta\uparrow 1}\lambda_{n}^{\beta}(\gamma)=\lambda(\gamma).
\]
Moreover, the function $w(\cdot,\gamma)$ and the constant $\lambda(\gamma)$ solve the Bellman equation \eqref{eq:Bellman1}. 
\end{theorem}

\begin{proof}
Fix $\gamma<0$; the proof for $\gamma>0$ is similar. Recalling that $E$ is finite and using propositions~\ref{pr:lambda.n.bound} and \ref{pr:w.disc.bound}, we know that there exists an increasing sequence $(\beta_k)_{k\in\bN}$, a sequence of functions $(w_n(\cdot,\gamma))_{n\in\bN}$, where $w_n(\cdot,\gamma)\colon E\to \bR$, and a sequence of constants $(\lambda_n(\gamma))_{n\in\bN}$, where $\lambda_n(\gamma)\in\bR$, such that $\beta_k\to 1$, as $k\to\infty$, and for each $n\in\bN$, we have
\[
\lambda_n^{\beta_k}(\gamma)\to\lambda_n(\gamma)\quad\textrm{and}\quad \bar w_n^{\beta_k}(x,\gamma)\to w_n(x,\gamma),\quad \textrm{as}\quad k\to\infty.
\]
Letting $k\to\infty$ and using \eqref{eq:bellman.vanishing2}, we find that for each $n\in\bN$, we have
\begin{equation}\label{eq:limit.bellman}
\textstyle w_n(x, \gamma)=\min_{a\in U}\left[\gamma c(x,a)-\lambda_n(\gamma)+\ln\sum_{y\in E}e^{ w_{n+1}(y,\gamma)}\bP^a(x,y)\right].
\end{equation}
Moreover, for each $n\in\bN$, using Proposition~\ref{pr:w.disc.bound}, Proposition~\ref{pr:lambda.n.bound}, and recalling that $\bar w^{\beta}(\bar z,\gamma)=0$, for any $\beta\in (0,1)$, we get 
\begin{align}
w_n(\bar z, \gamma) & =0,\label{eq:vanishing4}\\
|w_n(\cdot,\gamma)|& \leq |\gamma| \|c\| 2N+\ln K,\label{eq:vanishing5}\\
|\lambda_n(\gamma)|& \leq |\gamma| \|c\| (2N+1)+ \ln K.\label{eq:vanishing6}
\end{align}
For $n\in \bN$, let $\tilde w_n(\cdot,\gamma):=\tfrac{1}{\gamma}w_n(\cdot,\gamma)$ and $\tilde\lambda_n(\gamma):=\tfrac{1}{\gamma}\lambda_n(\gamma)$. Then, we can rewrite \eqref{eq:limit.bellman} as
\begin{equation}\label{eq:limit.bellman2}
\textstyle \gamma \tilde w_n(x, \gamma)=\min_{a\in U}\left[\gamma c(x,a)-\gamma\tilde\lambda_n(\gamma)+\ln\sum_{y\in E}e^{\gamma\tilde w_{n+1}(y,\gamma)}\bP^a(x,y)\right]
\end{equation}
or, using the long-run risk sensitive averaged Bellman operator notation introduced in \eqref{eq:Bellman.op}, as
\begin{equation}\label{eq:vanishing7}
\tilde w_n(x,\gamma)=T_{\gamma}\tilde w_{n+1}(x,\gamma)+\tilde\lambda_n(\gamma).
\end{equation}
Now, due to \eqref{eq:vanishing4}--\eqref{eq:vanishing6}, we get that
\[
\sup_{n\in\bN}|\tilde w_n(\cdot,\gamma)|\leq \|c\| 2N+\tfrac{1}{|\gamma|}\ln K\quad \textrm{and}\quad \sup_{n\in\bN}|\tilde\lambda_n(\gamma)|\leq  \|c\|(2N+1)+\tfrac{1}{|\gamma|}\ln K,
\]
which implies joint boundedness of the family of functions $(\tilde w_n(\cdot,\gamma))_{n\in\bN}$ in the span-norm. Using Theorem~\ref{th:local.contraction}, i.e.\ the local contraction property of the operator $T_{\gamma}$, and Banach's fixed point theorem, one can show that there exists a function $w(\cdot,\gamma)\colon E\to\bR$ such that for every $n\in\bN$ we get
\begin{equation}\label{eq:banach1}
\|\tilde w_n(\cdot,\gamma)-w(\cdot,\gamma)\|_{\textnormal{sp}}=0.
\end{equation}
To show \eqref{eq:banach1}, first observe that there exists $w(\cdot,\gamma)$ such that for each $n\in\bN$ we get $\|T^k_{\gamma}\tilde w_n(\cdot,\gamma)-w(\cdot,\gamma)\|_{\textnormal{sp}}\to 0$ as $k\to\infty$; function $w(\cdot,\gamma)$ clearly does not depend on $n\in\bN$ due to \eqref{eq:vanishing7}. Also, for any $k\in\bN$, we get $\|T^k_{\gamma}\tilde w_{n+k}(\cdot,\gamma)-\tilde w_n(\cdot,\gamma)\|_{\textnormal{sp}}=0$, and, by the boundedness of family $(\tilde w_n(\cdot,\gamma))_{n\in\bN}$, $\|T^k_{\gamma}\tilde w_{n+k}(\cdot,\gamma)- w(\cdot,\gamma)\|_{\textnormal{sp}}\to 0$, as $k\to\infty$. Consequently, we have
\[
\|\tilde w_n(\cdot,\gamma)-w(\cdot,\gamma)\|_{\textnormal{sp}}\leq \|\tilde w_n(\cdot,\gamma)-T_{\gamma}^{k}\tilde w_{n+k}(\cdot,\gamma)\|_{\textnormal{sp}}+\|T_{\gamma}^{k}\tilde w_{n+k}(\cdot,\gamma)-w(\cdot,\gamma)\|_{\textnormal{sp}}\to 0,
\]
as $k\to\infty$, which proves \eqref{eq:banach1}. Next, since $w(\cdot,\gamma)$ is defined up to an additive constant, we can assume, without loss of generality, that $w(\bar z,\gamma)=0$, which implies $w(\bar z,\gamma)=\tilde w_n(\bar z,\gamma)$, for $n\in\bN$, due to \eqref{eq:vanishing4}. Thus, due to \eqref{eq:banach1}, we get $\tilde w_n(\cdot,\gamma)\equiv w(\cdot,\gamma)$, that is, the functions $\tilde w_n(\cdot,\gamma)$ and $w(\cdot,\gamma)$ are the same. In particular, for any $n\in\bN$ we can rewrite \eqref{eq:vanishing7} as
\begin{equation}\label{eq:vanishing8}
w(x,\gamma)=T_{\gamma} w(x,\gamma)+\tilde\lambda_n(\gamma),
\end{equation}
which proves that the sequence $(\lambda_n(\gamma))_{n\in\bN}$ is constant. Let $\lambda(\gamma):=\lambda_1(\gamma)$. It is easy to observe that $w(\cdot,\gamma)$ and $\lambda(\gamma)$ is a unique solution to a renormalized averaged Bellman equation \eqref{eq:Bellman1} given by
\[
\gamma w(x,\gamma)=\min_{a\in U}\left[\gamma c(x,a)-\gamma \lambda(\gamma)+\ln\sum_{y\in E}e^{\gamma w(y,\gamma)}\bP^a(x,y)\right],
\]
which additionally satisfies $w(\bar z,\gamma)=0$.
Finally, since any generic increasing sequence $(\beta_m)_{m\in\bN}$, where $\beta_m\to 1$, as $m\to\infty$, contains a subsequence $(m_k)_{k\in\bN}$ such that
\[
\tfrac{1}{\gamma}\bar w^{\beta_{m_k}}(x,\beta^n_{m_k}\gamma)\to w(x,\gamma)\quad\textrm{and}\quad \tfrac{1}{\gamma}\lambda^{\beta_{m_k}}_n(\gamma)\to \lambda(\gamma),
\]
as $k\to\infty$, we know that $\bar w^{\beta}(x,\beta^n\gamma)\to\gamma w(x,\gamma)$ and $\lambda_n^{\beta}(\gamma)\to\lambda(\gamma)$ as $\beta\to 1$, i.e., those limits exist and are well defined. Note that from \eqref{eq:vanishing8} we also directly get that $w(\cdot,\gamma)$ and $\lambda(\gamma)$ solve the long-run risk-sensitive averaged Bellman equation~\eqref{eq:Bellman1}, which concludes the proof.
\end{proof}

\subsection{Blackwell property for the discounted risk-sensitive policies}
While for $\gamma=0$ we can always find stationary Markov solution to problem \eqref{eq:RSC.disc}, for $\gamma\neq 0$ this is not the case. That being said, one expects that some form of the Blackwell property holds in the risk-sensitive framework due to the vanishing discount framework, that is, the convergence results stated in Theorem~\ref{th:vanishing.discount}.

To show this, consider the optimal policy induced by the discounted Bellman equation \eqref{eq:Bellman2}. Namely, for any $\beta\in (0,1)$ and $\gamma\neq 0$, let $\hat\pi^{\beta}:=(\hat{u}^{\beta}_0,\hat{u}^{\beta}_1,\ldots,)\in\Pi^{*}_{\gamma}(\beta)$ be the optimal Markov policy given by 
\begin{equation}\label{eq:discounted.stat}
\textstyle \hat u^{\beta}_n(x):=\argmax_{a\in U}\left[c(x,a) + \frac{1}{\gamma\beta^n}\ln \left( \sum_{y\in E} e^{w^\beta(y,\gamma \beta^{n+1})}\mathbb{P}^a(x,y)\right)\right],
\end{equation}
where $w^{\beta}(\cdot,\cdot)$ is a solution to~\eqref{eq:Bellman2}. Note that the policy $\hat\pi^{\beta}$ could be non-stationary as we might have $\hat u^{\beta}_n\not\equiv \hat u^{\beta}_m$, for $n\neq m$; see Example~\ref{ex:non-stat}. We are now ready to present the analogue of Theorem~\ref{th:blackwell} transferred to the risk-sensitive framework, which might be interpreted as the Blackwell property of the risk-sensitive optimal strategies.

\begin{theorem}\label{th:blackwell.sensitive}
Assume \eqref{A.1} and \eqref{A.2}; fix $\gamma\neq 0$. Then, for any $n\in\bN$ there is $\beta_n\in (0,1)$ such that for any $\beta\in(\beta_n,1)$ the stationary Markov policy $\hat u^{\beta}_n\in \Pi'$ defined in \eqref{eq:discounted.stat} is optimal for the averaged problem, that is, we have $\hat u^{\beta}_n\in \Pi^*_{\gamma}$. 
\end{theorem}

\begin{proof}
Fix $\gamma<0$; the proof for $\gamma>0$ is similar. Assume that there exists $n\in\bN$, for which the claim of Theorem~\ref{th:blackwell.sensitive} is not true. Then, there exists an increasing sequence $(\beta_m)_{m\in\bN}$ of numbers in $(0,1)$ such that $\beta_m\to 1$, as $m\to\infty$, and a sequence $(\hat u^{\beta_m}_n)_{m\in\bN}$ such that, for every $m\in\bN$, the Bellman equation \eqref{eq:Bellman1} is not satisfied. Namely, we get that for, any $m\in\bN$, the renormalized Bellman equation
\begin{equation}\label{eq:blackwell.not1}
\textstyle w(x,\gamma)=\gamma c(x,\hat u^{\beta_m}_n(x))-\lambda(\gamma)+\ln\sum_{y\in E}e^{w(y,\gamma)}\bP^{\hat u^{\beta_m}_n(x)}(x,y)
\end{equation}
does not hold for at least one $x\in E$, where $\tfrac{1}{\gamma} w(\cdot,\gamma)$ and $\tfrac{1}{\gamma}\lambda(\gamma)$ is a unique solution to \eqref{eq:Bellman1} satisfying $w(\bar z,\gamma)=0$. Now, we can choose an increasing subsequence $(\beta_{m_k})_{k\in\bN}$ such that for some $\hat u\in\Pi'$ we have $\hat u^{\beta_{m_k}}_n(\cdot)\to \hat u(\cdot)$, as $k\to\infty$. Furthermore, since $U$ and $E$ are finite, there exists $k_0\in\bN$ such that $\hat u^{\beta_{m_k}}_n\equiv \hat u$ for $k\geq k_0$. Thus, recalling \eqref{eq:bellman.vanishing2}, we know that for any $k\geq k_0$ we have
\[
\textstyle \bar w_n^{\beta_{m_k}}(x, \gamma)=\gamma\beta_{m_k}^n c(x,\hat u(x))-\lambda_n^{\beta_{m_k}}(\gamma)+\ln\sum_{y\in E}e^{\bar w_{n+1}^{\beta_{m_k}}(y,\gamma)}\bP^{\hat u(x)}(x,y),\quad x\in E.
\]
Letting $k\to\infty$ and using Theorem~\ref{th:vanishing.discount}, we get
\[
\textstyle w(x, \gamma)=\gamma c(x,\hat u(x))-\lambda(\gamma)+\ln\sum_{y\in E}e^{ w(y,\gamma)}\bP^{\hat u(x)}(x,y),\quad x\in E,
\]
which, together with equality $\hat u\equiv u^{\beta_{m_{k_0}}}_n$, contradicts the fact that \eqref{eq:blackwell.not1} should not hold for at least one $x\in E$. This concludes the proof.
\end{proof}
From Theorem~\ref{th:blackwell.sensitive} we see that one can recover a long-run average optimal policy from the discounted setup by setting $n=1$ and recovering $\hat u^{\beta}_1(\cdot)$ for sufficiently large $\beta$. Although this might seem counterintuitive at first sight, one should expect that for $\beta$ close to $1$ the discounted optimal policy stabilizes. For instance, assuming there is a unique optimal policy for the long-run averaged problem, for large $\beta$s we get $u_1^{\beta}\equiv u$ and one expects
\[
\hat \pi^{\beta} \to (u^{\beta}_1,u^{\beta}_1,u^{\beta}_1,u^{\beta}_1,\ldots),\quad \textrm{as }\quad \beta\to 1,
\]
as ensured by Theorem~\ref{th:blackwell.sensitive}. See Example~\ref{ex:non-stat}, where this is illustrated. It is also worthwhile to note that the ideas presented in the proofs of Theorem~\ref{th:vanishing.discount} and Theorem~\ref{th:blackwell.sensitive} can be used to prove Theorem~\ref{th:blackwell}, i.e. our approach can be adapted to the risk-neutral case to recover the classical Blackwell property. For completeness, we present the sketch of the proof for the risk-neutral case in Appendix~\ref{S:Blackwell.neutral}.

\begin{remark}[Relation between Blackwell property and ultimate stationarity]\label{rem:ultimate.stationarity}
In \cite{Jaq1973,Jaq1976}, the concept of ultimate stationarity has been introduced and studied in reference to optimal strategies for the discounted long run risk-sensitive problem. A policy $\pi^{\beta}$ is {\it ultimately stationary} if it can represented as  $\pi^{\beta}=({u}^{\beta}_0,{u}^{\beta}_1,\ldots, u^{\beta}_{N},{u}^{\beta}_{\infty},{u}^{\beta}_{\infty},\ldots)$, for some Markov stationary policy $\hat u^{\beta}_{\infty}\in \Pi'$ and $N\in\bN$. Furthermore, it has been shown that for  any $\beta<1$ and $\gamma>0$ there is an ultimately stationary policy that is optimal for $\beta$-discounted risk-sensitive problem. We note that this concept is different than the Blackwell property posted in this paper. In particular, while Theorem~\ref{th:blackwell.sensitive} is linked to initial steps policies, the results from \cite{Jaq1973,Jaq1976} is related to (future) tail step policies. Furthermore, note that the number of steps $N$ for the ultimate stationarity depends on the parameter $\beta\in (0,1)$ and when $\beta\to 1$, then one might get $N\to \infty$. Consequently, the ultimately stationary (tail) discounted policies might fail to be optimal for the averaged problem, see Example~\ref{ex:non-stat} where this is illustrated.
\end{remark}

\section{Examples}\label{S:examples} In this section, we provide a series of examples that complement the results of the paper. In particular, we provide a series of specific controlled MDP frameworks which discriminate policies in the risk-neutral and risk-sensitive case, show that the set $\Gamma(u)$ might not be compact, and illustrate the Blackwell property in the setting in which there is no stationary Markov policy in the discounted framework.

\begin{example}[Risk-neutral optimal policy that is not optimal in the risk-sensitive case]\label{ex:1}
Let $E=\{-1,0,1\}$, $U=\{1,2,3\}$, and let the (controlled) transition matrices be given by
\[
P_1=\begin{bmatrix}
0.1 & 0.8 & 0.1 \\
0.1 & 0.8 & 0.1\\
0.1 & 0.8 & 0.1
\end{bmatrix},\,
P_2=\begin{bmatrix}
0.2 & 0.6 & 0.2\\
0.2 & 0.6 & 0.2\\
0.2 & 0.6 & 0.2
\end{bmatrix},\,
P_3=\begin{bmatrix}
0.3 & 0.4 & 0.3\\
0.3 & 0.4 & 0.3\\
0.3 & 0.4 & 0.3
\end{bmatrix}.
\]
Note that both \eqref{A.1} and \eqref{A.2} are satisfied. Let the function $c\colon E\times U\to\bR$ be given by $c(x,a)=x$.  Then, we get
\[
\Gamma(u_1)=(-\infty,0],\quad \Gamma(u_2)=\{0\},\quad \Gamma(u_3)=[0,\infty),
\]
which shows that the optimal risk-neutral policy does not need to be optimal for the risk-sensitive case with small risk aversion. To show this, let us first note that in the risk-neutral case a single-step expected reward is always equal to zero, regardless of the action taken, so that $\Pi^*_{0} \supset \Pi'$. This shows that every Markov policy is optimal. In particular, this refers to a constant Markov policy $u_2\equiv 2$. On the other hand, for $\gamma\neq 0$, we can find optimal policy by analysing the Bellman equation \eqref{eq:Bellman1} given by
\[
\textstyle w(x,\gamma)=x-\lambda(\gamma) +\sup_{a\in U}\left[\frac{1}{\gamma}\ln\sum_{y\in E}e^{\gamma w(y,\gamma)}\bP^a(x,y)\right].
\]
Setting
\begin{equation}\label{eq:K.ex1} 
\textstyle  K^a(\gamma):=\frac{1}{\gamma}\ln\sum_{y\in E}e^{\gamma w(y,\gamma)}\bP^a(x,y),\quad \gamma\neq 0,
\end{equation}
 observing that \eqref{eq:K.ex1} is independent of $x$, and assuming (without loss of generality) that $w(0,\gamma)=0$, we get $\lambda(\gamma)=\sup_{a\in U}K^a(\gamma)$ as well as $w(1,\gamma)=1$ and $w(-1,\gamma)=-1$, as  $\textstyle w(x,\gamma)-x$ is a constant. Thus, noting that
\[
K^a(\gamma)= \begin{cases}
\frac{1}{\gamma}\ln\left[1e^{-\gamma}+8e^{0}+1e^{\gamma}\right]-\frac{\ln 10}{\gamma}& \textrm{if } a=1,\\
\frac{1}{\gamma}\ln\left[2e^{-\gamma}+6e^{0}+2e^{\gamma}\right]-\frac{\ln 10}{\gamma}& \textrm{if } a=2,\\
\frac{1}{\gamma}\ln\left[3e^{-\gamma}+4e^{0}+3e^{\gamma}\right]-\frac{\ln 10}{\gamma}& \textrm{if } a=3.
\end{cases}
\]
and introducing $g(x):=e^{-x}+e^{x}$, we can rewrite \eqref{eq:K.ex1} as
\begin{align*}
K^a(\gamma)& =\tfrac{1}{\gamma}\ln\left[ag(\gamma)+(5-a)g(0)\right] -\tfrac{\ln 10}{\gamma}\\
& =\tfrac{1}{\gamma}\ln\left[a(g(\gamma)-g(0))+5g(0)\right] -\tfrac{\ln 10}{\gamma}.
\end{align*}
As $g(\gamma)>g(0)>1$, we get $\gamma K^1(\gamma) <\gamma K^2(\gamma)<\gamma K^3(\gamma)$, for any $\gamma\neq 0$. This implies that for $\gamma\leq 0$ the strategy $u_1\equiv 1$ is optimal, while for $\gamma\geq 0$ the strategy $u_3\equiv 3$ is optimal. For completeness, in Figure \ref{fig:EX01} we provide a graphical illustration of mapping $\gamma \to \lambda^u(\gamma)$ for the considered policies.
\end{example}

\begin{example}[Non-compactness of the set $\Gamma(u)$]\label{ex:2}
Let $E=\{0,1,2,3\}$, $U=\{1,2\}$ and let the (controlled) transition matrices be given by
\[
P_1=\begin{bmatrix}
0.2 & 0.1  & 0.5 & 0.2 \\
0.2 & 0.1  & 0.5 & 0.2 \\
0.2 & 0.1  & 0.5 & 0.2 \\
0.2 & 0.1  & 0.5 & 0.2
\end{bmatrix},\quad
P_2=\begin{bmatrix}
0.1 & 0.5  & 0.1 & 0.3 \\
0.1 & 0.5  & 0.1 & 0.3 \\
0.1 & 0.5  & 0.1 & 0.3 \\
0.1 & 0.5  & 0.1 & 0.3 
\end{bmatrix}.
\]
Note that both \eqref{A.1} and \eqref{A.2} are satisfied. Let the function $c\colon E\times U\to\bR$ be given by $c(x,a)=x$. Following similar reasoning as in Example~\ref{ex:1}, for $u_1\equiv 1$ and $u_2\equiv 2$ we get
\[
\Gamma(u_1)=\left[\ln\tfrac{3-\sqrt{5}}{2},\ln\tfrac{3+\sqrt{5}}{2}\right],\quad \Gamma(u_2)=\left(-\infty,\ln\tfrac{3-\sqrt{5}}{2}\right] \cup \left[\ln\tfrac{3+\sqrt{5}}{2},\infty\right),
\]
which shows that the set $\Gamma(u)$ does not need to be compact. For completeness, in Figure \ref{fig:EX01} we provide a graphical illustration of mapping $\gamma \to \lambda^u(\gamma)$ for the considered policies.
\end{example}

\begin{example}[Existence of isolated optimal strategies]\label{ex:3}
Let $E=\{-2,-1,1,2\}$, $U=\{1,2\}$ and let the (controlled) transition matrices by given by
\[
P_1=\begin{bmatrix}
0.2 & 0.3  & 0.3 & 0.2 \\
0.2 & 0.3  & 0.3 & 0.2 \\
0.2 & 0.3  & 0.3 & 0.2 \\
0.2 & 0.3  & 0.3 & 0.2
\end{bmatrix},\quad
P_2=\begin{bmatrix}
0.1 & 0.5  &0.1 & 0.3 \\
0.1 & 0.5  &0.1 & 0.3 \\
0.1 & 0.5  &0.1 & 0.3 \\
0.1 & 0.5  &0.1 & 0.3
\end{bmatrix}.
\]
Note that both \eqref{A.1} and \eqref{A.2} are satisfied. Let the function $c\colon E\times U\to\bR$ be given by $c(x,a)=x$. Following similar reasoning as in Example~\ref{ex:1}, for $u_1\equiv 1$ and $u_2\equiv 2$ we get
\[
\Gamma(u_1)=\{0\},\quad \Gamma(u_2)=\bR,
\]
which shows that there could exist an isolated optimal strategy. For completeness, in Figure \ref{fig:EX01} we provide a graphical illustration of mapping $\gamma \to \lambda^u(\gamma)$ for the considered policies.
\end{example}

\begin{figure}[htp!]
    \centering
\includegraphics[width=0.33\textwidth]{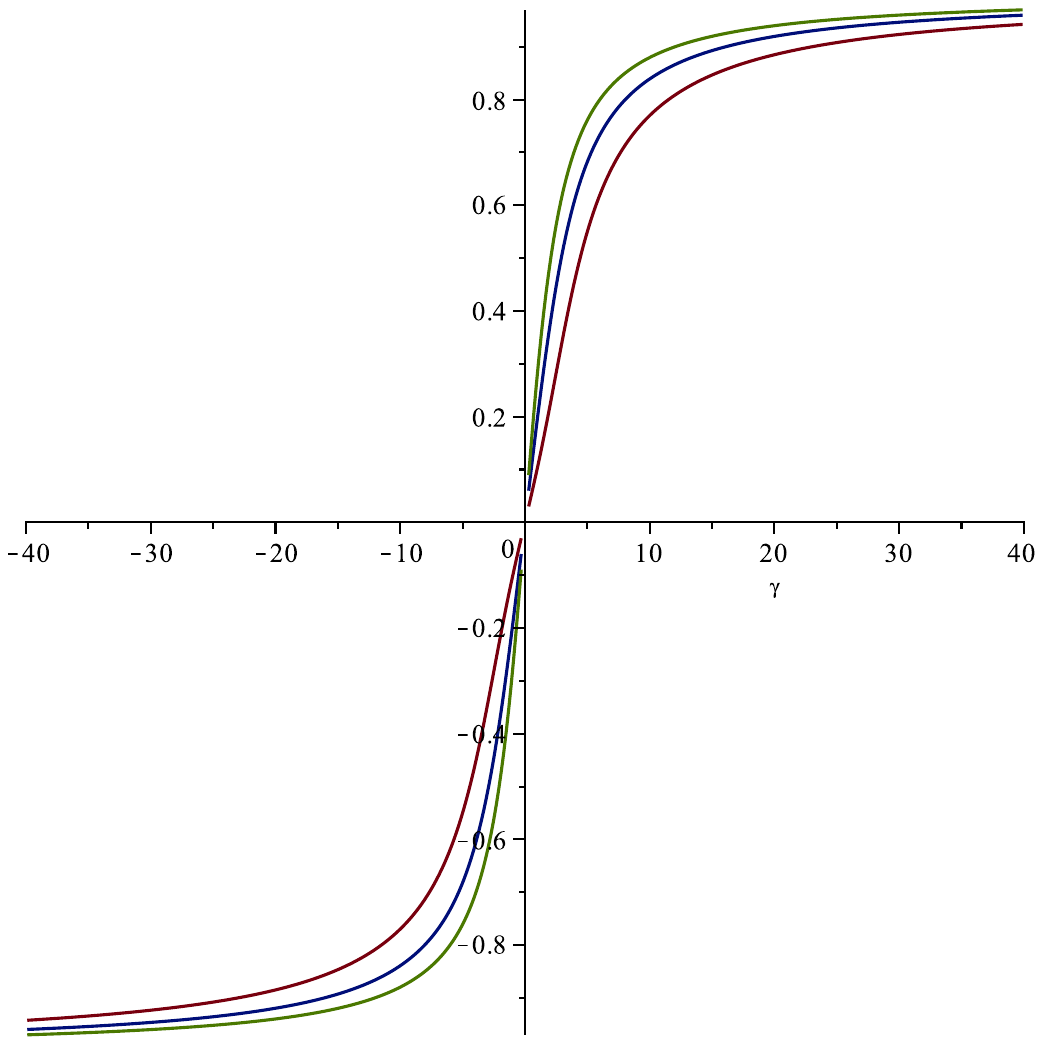}\includegraphics[width=0.33\textwidth]{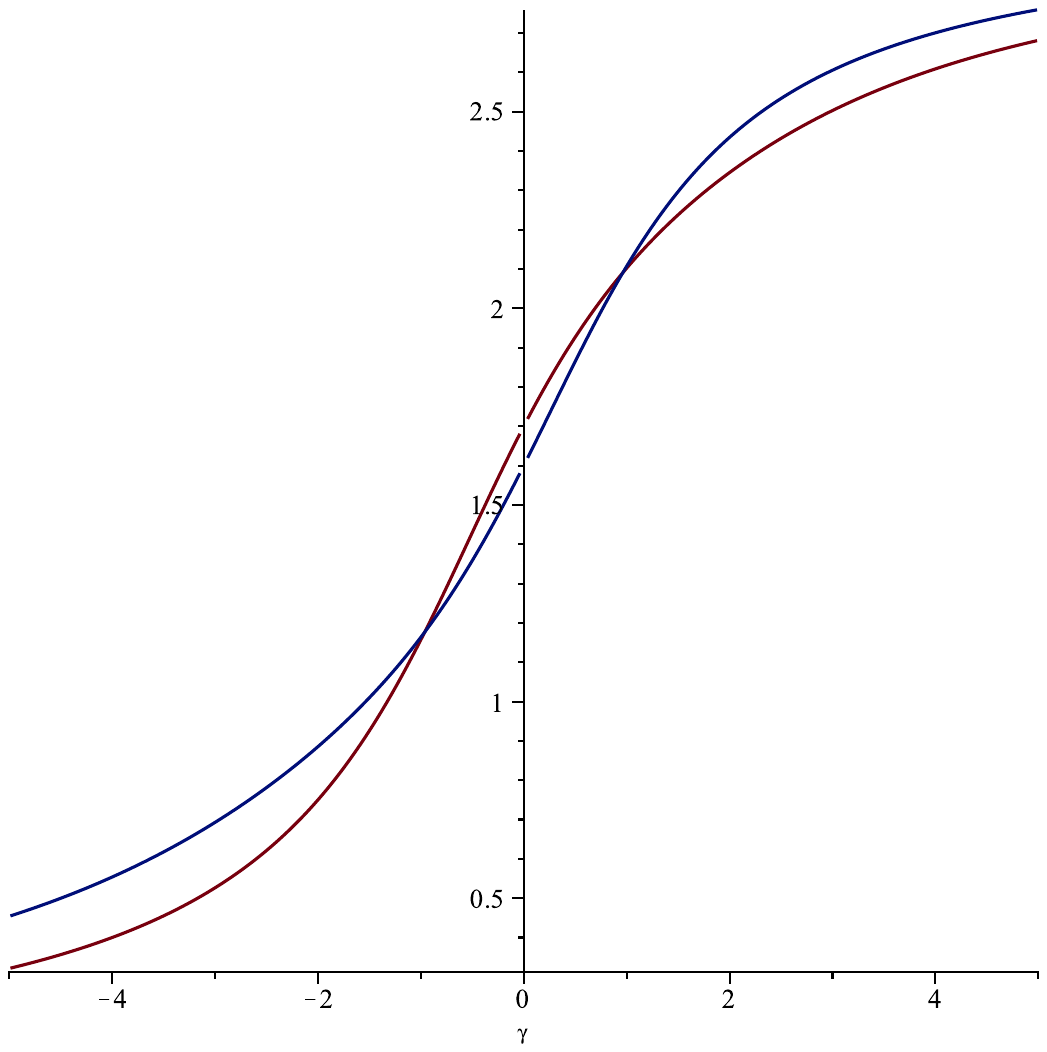}\includegraphics[width=0.33\textwidth]{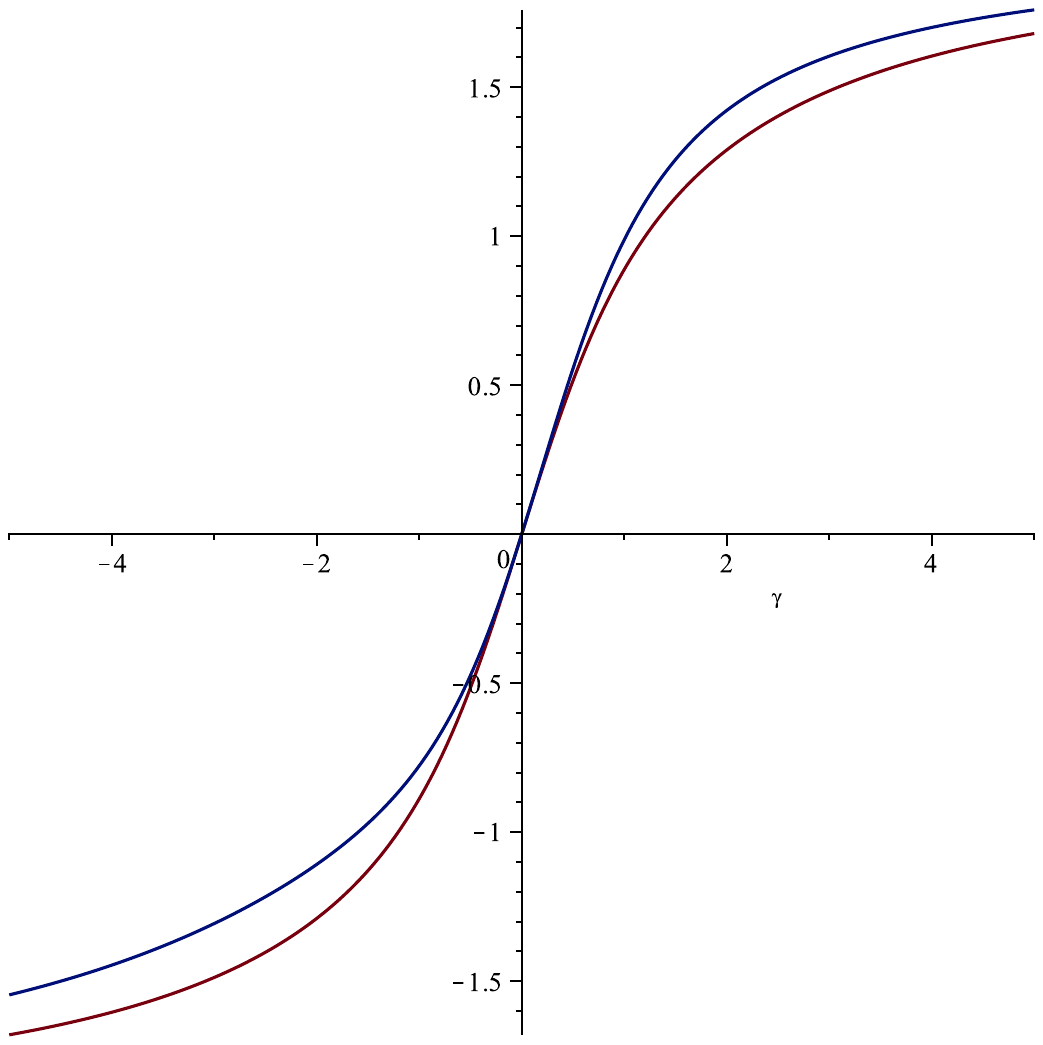}\vspace*{-1cm}
    \caption{Graphs of $\gamma \to \lambda^u(\gamma)$ for MDP specifications from Example~\ref{ex:1} (left), Example~\ref{ex:2} (middle), and Example~\ref{ex:3} (right). The performance of various Markov stationary policies is illustrated: $u_1\equiv 1$ (red), $u_1\equiv 2$ (blue), $u_3\equiv 3$ (green); note that policy $u_3$ applies only to Example~\ref{ex:1}.}
    \label{fig:EX01}
\end{figure}

\begin{example}[Non-stationary optimal policy for the discounted problem and Blackwell optimality]\label{ex:non-stat}
Let $E=\{1,2,3\}$, $U=\{0,1\}$ and let the (controlled) transition matrices be given by
\[
P_1=\begin{bmatrix}
2\epsilon & 0.5-\epsilon  & 0.5-\epsilon \\
1 & 0  & 0 \\
1 & 0  & 0 \\
\end{bmatrix},\quad
P_2=\begin{bmatrix}
2\epsilon & 0.9-\epsilon  & 0.1-\epsilon \\
1 & 0  & 0 \\
1 & 0  & 0 \\
\end{bmatrix},
\]
where $\epsilon\in (0,0.1)$. Let the function $c\colon E\times U\to\bR$ be given by
\[
c(x,1)=
\begin{cases}
0& \textrm{if } x=1\\
0& \textrm{if } x=2\\
8& \textrm{if } x=3
\end{cases}\quad\textrm{and}\quad 
c(x,2)=
\begin{cases}
1& \textrm{if } x=1\\
0& \textrm{if } x=2\\
8& \textrm{if } x=3
\end{cases}.
\]
This could be seen as a modification of the example presented in~\cite{Jaq1976} in which the introduction of $\epsilon>0$ ensures that assumptions \eqref{A.1} and \eqref{A.2} are satisfied. In this example, we have eight Markov stationary policies. Noting that $p_1(2,\cdot)\equiv p_2(2,\cdot)$, $p_1(3,\cdot)\equiv p_3(2,\cdot)$,  $c(2,\cdot)\equiv 0$, and $c(3,\cdot)\equiv 8$, we conclude that it does not matter which action is assigned to state $2$ and $3$. Consequently, without loss of generality, it is sufficient to consider two Markov policies given by $u(\cdot):=0$ and $\tilde u(\cdot):=1$. 

First, let us show that for sufficiently small $\epsilon>0$ a non-stationary policy is optimal. In fact, it is sufficient to consider the limit case $\epsilon=0$. This is due to the fact that the objective value function is continuous with respect to $\epsilon$ for any fixed policy and because the optimal non-stationary policy (for $\epsilon=0$) will give value strictly larger than any stationary policy; for brevity, we omit detailed proof of this fact (one can control the value for the first $N\in\bN$ sum constituents by modifying the size of $\epsilon$, while the tail sum total value can be controlled by the choice of $N$ due to the discounting scheme). Thus, fixing $\epsilon=0$, initial state $x=1$, following~\cite{Jaq1976}, and noting that for both $u$ and $\tilde u$, the corresponding Markov processes can be seen as a sequence of independent (2-step) gambles, we know that for any $\gamma\in\bR\setminus\{0\}$ and $\beta\in (0,1)$ we get
\begin{align*}
J_{\gamma}(1,u;\beta) &=\ent_1^{u}\left(\sum_{i=0}^{\infty}\beta^{(2i+1)}8 Z_i,\gamma\right) &=\sum_{i=0}^{\infty}\frac{1}{\gamma}\ln\left(0.5e^0+0.5e^{\gamma\beta^{2i+1}8} \right),&\\
J_{\gamma}(1,\tilde u;\beta)& =\ent_1^{\tilde u}\left(\sum_{i=0}^{\infty}\beta^{(2i+1)}8 Z_i,\gamma\right)+\sum_{i=0}^{\infty}\beta^{(2i)} &=\sum_{i=0}^{\infty}\frac{1}{\gamma}\ln\left(0.9e^0+0.1e^{\gamma\beta^{2i+1}8} \right),&
\end{align*}
where $(Z_n)_{n\in\bN}$ is a sequence of i.i.d. random variables such that $Z_n \sim B(1,0.5)$ under $\bP_{1}^u$ and $Z_n \sim B(1,0.1)$ under $\bP_{1}^{\tilde u}$; recall that entropic utility is additive with respect to independent random variables (and constants). Moreover, using the similar logic, we get that the value function for the non-stationary policy $\pi:=(\tilde u, u, u,\ldots)$ is equal to 
\[
\textstyle J_{\gamma}(1,\pi;\beta) =\frac{1}{\gamma}\ln\left(0.9e^0+0.1e^{\gamma\beta 8} \right)+1+\sum_{i=1}^{\infty}\frac{1}{\gamma}\ln\left(0.5e^0+0.5e^{\gamma\beta^{2i+1}8} \right).
\]
In particular, fixing $\gamma=-1$ and $\beta=\tfrac{1}{2}$, by direct computations, we get
\[
J_{-1}(1,u;\tfrac{1}{2})\approx 1.21,\quad  J_{-1}(1,\tilde u;\tfrac{1}{2})\approx 1.53,\quad \textrm{and}\quad J_{-1}(1,\pi;\tfrac{1}{2})\approx 1.64,
\]
which shows that the objective value over non-stationary policies is strictly bigger than the value of any stationary policy. Furthermore, we can recover the optimal non-stationary policy by confronting (individually) sum constituents size resulting from applying $\tilde u$ and $u$ in individual steps. Noting that the positive real root of the function $f(s):=0.5e^{0}+0.5e^{-4s}-0.9e^{-s}-0.1e^{-5s}$ is in the point $s\approx 0.456$ and $\beta^{2\cdot 1}=0.25 <s$, we find that it is optimal to apply $\tilde u$ only in the first step, and then follow $u$; this shows that policy $\pi$ is in fact optimal in the class of all non-stationary Markov policy, for $\gamma=-1$ and $\beta=\tfrac{1}{2}$.

Second, let us provide the link between the considered example and the Blackwell optimality property. Using similar logic, one can show that for $\gamma=-1$ and any $\beta>\tfrac{1}{2}$, the optimal policy is non-stationary and of the form $\pi=(\tilde u,\ldots,\tilde u, u,u,\ldots)$, where the policy change point depends on $\beta$ and includes more (initial) executions of policy $\tilde u $ for large values of $\beta$. In particular, we see that in the limit case ($\beta\to 1$), the policy $\tilde u$ should be optimal. This is indeed the case, as for the averaged criterion we get
\begin{align*}
J_{-1}(1,u;1) &=\lim_{n\to\infty}\frac{1}{2n}\sum_{i=0}^{n-1}\ent_1^{u}\left(8 Z_i,\gamma\right)=-\tfrac{1}{2} \ln\left(0.5e^0+0.5e^{-8}\right)\approx 0.35,\\
J_{-1}(1,\tilde u;1) &=\lim_{n\to\infty}\frac{1}{2n}\sum_{i=0}^{n-1}\ent_1^{\tilde u}\left(8 Z_i+1,\gamma\right)=-\tfrac{1}{2} \ln\left(0.9e^0+0.1e^{-8}\right)+\tfrac{1}{2}\approx 0.67.\end{align*}
To sum up, we see that when $\beta\to 1$, then the policy $\tilde u$ becomes Blackwell optimal, eventually for any fixed execution step, as indicated by Theorem~\ref{th:blackwell.sensitive}. Note that this example also shows that the ultimately stationary tail policy (i.e. $u$) might fail to be optimal in the average case, as stated in Remark~\ref{rem:ultimate.stationarity}.
\end{example}

\appendix
\section{Alternative proof of the classical Blackwell property}\label{S:Blackwell.neutral}
In this section we provide a sketch of an alternative proof of the classical Blackwell property, i.e. Theorem~\ref{th:blackwell}, that is based on argumentation presented in Theorem~\ref{th:vanishing.discount} and Theorem~\ref{th:blackwell.sensitive}; the proof does not rely on the analytical/polynomial properties of the underlying eigenvalues or the Laurent series expansion, see also \cite{cavazos1999direct}.

\medskip
\begin{proof}[Proof sketch (Theorem~\ref{th:blackwell})]  Let us fix $\gamma=0$. From Theorem~\ref{th:exist} we know that under \eqref{A.1} there exists a unique solution (up to an additive constant) to the risk-neutral Bellman equation~\eqref{eq:Bellman1}, i.e. we can solve equation
\begin{equation}\label{eq:app0}
\textstyle w(x,0)+\lambda(0)=\max_{a\in U}\left[c(x,a)+\sum_{y\in E} w(y,0)\bP^a(x,y)\right]; \tag{AP.1}
\end{equation}
note that for risk-neutral case we do not require assumption \eqref{A.2}. Furthermore, one can show that for any $\beta\in (0,1)$ there always exists a stationary solution to the risk-neutral discounted Bellman equation \eqref{eq:Bellman2}, that is, there is $w^{\beta}$ such that
\begin{equation}\label{eq:app1}
\textstyle w^{\beta}(x,0)=\max_{a\in U}\left[c(x,a)+\beta\sum_{y\in E} w^{\beta}(y,0)\bP^a(x,y)\right].
\tag{AP.2}
\end{equation}
Under \eqref{A.1}, we get that $\| w_{\beta}\|_{span}\leq 2\|c\|/(1-\beta\Delta)\leq 2\|c\|/(1-\Delta)$, so that the family $(w^{\beta})_{\beta\in (0,1)}$ is uniformly bounded w.r.t.\ ~$\beta$, cf. Proposition~\ref{pr:w.disc.bound}. We can also consider the modified (centered) solution given by $\bar w_{\beta}(x):=w_{\beta}(x)-w_{\beta}(\bar z)$ for some $\bar z\in E$, and apply the vanishing discount proof approach from Theorem~\ref{th:vanishing.discount}. Namely, noting that we have
\begin{equation}\label{eq:app2}
\textstyle \bar w^{\beta}(x,0)+(1-\beta)w^{\beta}(\bar z,0)=\max_{a\in U}\left[c(x,a)+\beta\sum_{y\in E} \bar w^{\beta}(y,0)\bP^a(x,y)\right],
\tag{AP.3}
\end{equation}
and considering appropriate sequence $\beta_n\uparrow 1$, we can recover \eqref{eq:app0} from \eqref{eq:app2} with $\bar w^{\beta_n}(x)\to w(x)$ and $\lambda^{\beta_n}(0):=(1-\beta_n)w^{\beta_n}(\bar z,0) \to \lambda(0)$. Finally, recalling joint boundedness of $(\bar w_{\beta_n})_{n\in\bN}$, the fact that $\bar w^{\beta_n}(\bar z)=0$, and using arguments from Theorem~\ref{th:blackwell.sensitive}, we can recover the classical Blackwell property. Note that since all the solutions are stationary, the proof simplifies, i.e. it is sufficient to consider Theorem~\ref{th:blackwell.sensitive} for fixed $n=1$ for the risk-neutral case.
\end{proof}

\bibliographystyle{siamplain}
\bibliography{RSC_bibliografia}
\end{document}